\documentclass[11 pt]{amsart}

\usepackage[backend=biber,style=trad-alpha,sorting=nyt,natbib=true,firstinits=true,doi=false,isbn=false,url=false,eprint=false]{biblatex}
\AtEveryBibitem{%
  \ifentrytype{book}{%
    \clearfield{pages}%
  }{%
  }%
}

\usepackage{hhline}
\usepackage{amsmath,amssymb,amsthm,amsfonts}
\usepackage{accents}
\usepackage{mathtools}
\usepackage{pdftexcmds}
\usepackage{enumerate}
\usepackage{comment}
\usepackage{hyperref}
\usepackage{multirow}
\usepackage{color}
\usepackage{pst-node}

\allowdisplaybreaks
\mathtoolsset{showonlyrefs}

\DeclareMathOperator{\supp}{supp}

\DeclareMathOperator{\adj}{adj}

\DeclareMathOperator{\Diff}{Diff}

\DeclareMathOperator{\lf}{\ell f}
\newcommand{\rf}{r\text{f}}
\DeclareMathOperator{\ff}{ff}

\let\Re\undefined

\DeclareMathOperator{\Re}{Re}

\let \sectionsymbol \S

\newcommand{\R}{\mathbb{R}}
\newcommand{\B}{\mathbb{B}}
\newcommand{\C}{\mathbb{C}}
\renewcommand{\S}{\mathbb{S}}
\renewcommand{\H}{\mathbb{H}}

\newcommand{\p}{\partial}

\newcommand{\tX}{\widetilde{X}}
\newcommand{\ta}{{\widetilde{\alpha}}}
\newcommand{\tx}{{\widetilde{x}}}
\newcommand{\ty}{\widetilde{y}}
\newcommand{\oX}{\overline{X}}

\newcommand{\rg}{\rangle}
\renewcommand{\lg}{\langle}

\newcommand{\Di}{\Delta\iota}

\newcommand{\tPhi}{\widetilde{\Phi}}

\newcommand{\cM}{\mathring{M}}
\newcommand{\pM}{\partial M}

\newcommand{\calA}{\mathcal{A}}

\newcommand{\calE}{\mathcal{E}}
\newcommand{\calF}{\mathcal{F}}

\newcommand{\calL}{\mathcal{L}}
\newcommand{\calM}{\mathcal{M}}
\newcommand{\calN}{\mathcal{N}}

\newcommand{\calR}{\mathcal{R}}
\newcommand{\calS}{\mathcal{S}}
\newcommand{\calU}{\mathcal{U}}
\newcommand{\calV}{\mathcal{V}}
\newcommand{\calW}{\mathcal{W}}

\renewcommand{\a}{\alpha}
\renewcommand{\b}{\beta}
\newcommand{\g}{\gamma}
\renewcommand{\d}{\delta}
\let\epsilon\varepsilon
\newcommand{\e}{\epsilon}

\newcommand{\z}{\zeta}
\newcommand{\smsec}{\O_0^{{1}/{2}}}

\renewcommand{\r}{\rho}
\renewcommand{\t}{\tau}
\renewcommand{\k}{\kappa}
\renewcommand{\l}{\lambda}
\newcommand{\s}{\sigma}
\renewcommand{\th}{\theta}

\renewcommand{\O}{\Omega}

\newtheorem{theorem}{Theorem}

\newtheorem{definition}{Definition}
\numberwithin{definition}{section}

\newtheorem{proposition}[definition]{Proposition}

\newtheorem{lemma}[definition]{{Lemma}}

\newtheorem{remark}[definition]{Remark}

\newtheorem{corollary}[definition]{Corollary}

\numberwithin{equation}{section}

\let \o \undefined
\def \o#1{\overline{#1}}

\let\td\undefined
\def \td#1{\widetilde{#1}}
\let\implies\Rightarrow

\title[Stability Estimates for the X-Ray Transform]{Stability Estimates for the X-Ray Transform on simple Asymptotically Hyperbolic manifolds}

\newcommand{\oz}{{\overline{\zeta}}}

\let\pointyphi\phi 
\let \vphi\phi
\let\phi\varphi

\begin{document}

\begin{abstract}{We study the normal operator to the geodesic X-ray transform on functions in the setting of simple asymptotically hyperbolic manifolds. We construct a parametrix for the normal operator in the 0-pseudodifferential calculus and use it show a stability estimate.
}

\author[Nikolas Eptaminitakis]{Nikolas Eptaminitakis}
\address{Department of Mathematics, Purdue University\\
West Lafayette, IN 47907}
\email{neptamin@purdue.edu}

\end{abstract}

\maketitle

\section{Introduction}

In this paper we consider the stability of geodesic X-ray transform 
\begin{equation}\label{eq:geodesic_xray}
	If(\g)=\int_\g f \,ds,
\end{equation}
where $\g$ is a unit speed geodesic, in the setting of asymptotically hyperbolic manifolds.
Our goal is to establish that small perturbations of the X-ray transform  cannot originate from large perturbations of the unknown function $f$, and find  appropriate spaces to measure such perturbations.
Sharp stability estimates are known for the X-ray transform in $\R^n$ (see \cite{MR856916}) and stability has also been studied extensively on compact manifolds with boundary under a variety of assumptions (see e.g. \cite{MR0431074}, \cite{MR511273}, \cite{MR1374572}, \cite{MR2068966}, \cite{MR2365669}, \cite{Uhlmann2016}, \cite{MR3770848}, \cite{assylbekov2018sharp} and \cite{ilmavirta2018integral} for a survey).

Introducing our geometric setting, an $n+1$-dimensional Riemannian manifold $(\cM,g)$ will be called {asymptotically hyperbolic (AH)} if $\cM$ is the interior of a smooth compact manifold with  boundary $M$ such that for some (and thus any) boundary defining function $x\in C^\infty(M)$ it is the case that $\o{g}:=x^2g$ extends to a $C^\infty$ Riemannian metric on $M$ with $\|dx\big|_{\pM}\|^2_{\o{g}}\equiv 1$.
Recall that {$x$ is a boundary defining function if $x\big|_{\p M}\equiv0$, $dx\big|_{\pM}\not\equiv 0$ and $x>0$ on $\cM$.}
AH manifolds generalize the Poincar\'e model of hyperbolic space. An AH metric $g$ determines a conformal class of metrics on $\p M$, called the \textit{conformal infinity} and given by $[x^2g\big|_{T\p M}]$.
As shown in \cite{MR2941112}, AH manifolds are geodesically complete and their sectional curvatures approach $-\|dx\big|_{\pM}\|^2_{\o{g}}=- 1$ as $x\to 0$.
Moreover, any unit speed geodesic $\g(t)$ in an AH manifold which eventually exits every compact set approaches a boundary point as $t\to \infty$, orthogonally to the boundary, and $x\circ \g(t)=O(e^{-t})$.
One of the issues when studying the X-ray transform on AH manifolds is that due to completeness  the integral in \eqref{eq:geodesic_xray} might not converge unless some conditions are imposed on $\g$ and  $f$; for instance, provided $\g$ does not spend infinite time in any compact set $K\subset \cM$ (i.e. it is not \textit{trapped}), it suffices to assume that $f\in |\log x|^{\a}C^0(M) $, $\a<-1$.
Parameterizing the space of geodesics in an AH manifold is also a non-trivial task; this point will be discussed in more detail in Section \ref{sec:geodesic_flow}.
In this paper we are concerned with \textit{simple} AH manifolds, defined in \cite{2017arXiv170905053G}; those are by definition non-trapping (i.e. they contain no trapped geodesics) and they have no boundary conjugate points, that is, there exists no non-trivial Jacobi field $Y$ along any unit-speed geodesic such that $\lim_{t\to\pm\infty}|Y(t)|_{g}\to 0$.
Using work of Knieper in \cite{MR3825848}, the authors in \cite{2017arXiv170905053G} showed that on a simple AH manifold there exist no interior conjugate points in the usual sense, whereas in \cite{eptaminitakis2019asymptotically} it was shown that non-trapping AH manifolds with no conjugate points in the usual sense may have boundary conjugate points, i.e. they are not  simple in general.
A simple AH manifold is necessarily simply connected and diffeomorphic to a ball, via the exponential map at any point, though it might exhibit positive curvature.
Simple AH manifolds are a natural geometric setting for the study of inverse problems such as tensor tomography and boundary rigidity. The study of those problems in the setting of AH manifolds was initiated in \cite{2017arXiv170905053G}, and the present work is meant as a step in this direction.

Our approach to stability is mainly inspired by the works of Stefanov-Uhlmann \cite{MR2068966} and of Berenstein-Casadio Tarabusi (\cite{MR1104811}), both of which analyze the \textit{normal operator} to the X-ray transform in different settings.
On a simple compact manifold with boundary $(X,\td{g})$ (i.e. non-trapping, with no conjugate points and with strictly convex boundary), which is the setting of \cite{MR2068966}, the normal operator is given by $\calN_{\td{g}}=I^*I$, where
\begin{equation}
	I^*F(z)=\int_{S^*_zX}F(\xi)d\mu_{\td{g}}(\xi), \quad F\in{C}^\infty(S^*X),\; z\in X;
\end{equation}
 here $d\mu_{\td{g}}$ is the measure induced on each fiber of $S_z^*X$ by the Lebesgue measure on $T_z^*X$ and for $f\in C^\infty (X)$, $If$ is understood as a function on the unit cosphere bundle $S^*X:=\{(z,\xi)\in T^*X: |\xi|_{\td{g}}=1\}$ which is constant along the orbits of the geodesic flow.
For now the notation $I^*$ is formal, however $I^*$ can be interpreted as a formal adjoint for $I$ using suitable inner products and function spaces (this is discussed in Section \ref{sec:geodesic_flow} for the AH case).
In \cite{MR2068966} it was shown that 
 $\calN_{\td{g}}$ extends to an elliptic pseudodifferential operator of order $-1$ on $\tX$, where $\tX$ is an open domain slightly larger than $X$ and of the same dimension, such that its closure is still simple. 
The ellipticity of $\calN_{\td{g}}$ then implied the existence of a left parametrix (inverse up to compact error), and this yielded a stability estimate of the form 
\begin{equation}\label{eq:normal_operator_estimate_cpt}
	\|u\|_{L^2(X)}\leq C \|\calN_{\td{g}} u\|_{H^1(\td{X})},\quad u\in L^2(\td{X}),\:\supp u\subset X,
\end{equation} using injectivity of $I$ on simple manifolds, which had already been established in the '70s.
The construction of the normal operator carries over in the same way on hyperbolic space and it is well defined on $C^\infty$ functions of suitable decay at infinity;
the authors of \cite{MR1104811} derived explicit inversion formulas for it using the spherical Fourier transform for radial distributions on hyperbolic space (see \cite{MR1723736}).
Although they did not explicitly state a stability estimate,  the estimate of Theorem \ref{thm:main} below for the special case of hyperbolic space follows immediately from their work using the machinery of the 0-calculus, which we will discuss shortly.

In \cite{2017arXiv170905053G} it was shown that on a simple AH manifold $(\cM,g)$,  $I$ is injective on $x C^\infty(M)$, where $x\in C^\infty(M)$ is a boundary defining function (in fact it is shown there that one can allow for trapped geodesics as well, provided that the trapped set is hyperbolic for the geodesic flow).
The method of proof relied on showing that functions a priori in $x C^\infty(M)$ which lie in the nullspace of $I$ actually vanish to infinite order at $\p M$, and it does not yield stability.
The main result of the present work is a stability estimate analogous to \eqref{eq:normal_operator_estimate_cpt} on simple AH manifolds and a strengthened injectivity result.
The normal operator on a simple AH manifold $(\cM,g)$ is defined similarly to the case of simple compact manifolds with boundary: we let
\begin{equation}
	\calN_gf=I^*If(z)=\int_{S^*_z\cM}If(\xi)d\mu_g(\xi), \quad f\in \dot{C}^\infty(M),\; z\in \cM,
\end{equation}
where $\dot{C}^\infty(M)$ denotes smooth functions vanishing to infinite order at the boundary and $d\mu_g$ is the measure induced on the fibers of $S^*\cM$ by $g$, as before.
In our setting, $\calN_g$ turns out to be a well behaved object within the framework of the 0-calculus of pseudodifferential operators of Mazzeo and Melrose, which was introduced in \cite{MR916753} to analyze a modified resolvent of the Laplacian on AH spaces.
0-pseudodifferential operators generalize the \textit{0-differential operators}, consisting of finite sums of finite products of \textit{0-vector fields}: those are the smooth vector fields on $M$ that vanish on $\p M$ and are denoted by $\calV_0$.
They can be written locally near $\p M$ as smooth linear combinations of $\{x\p_x, x\p_{y^1},\dots, x\p_{y^n}\}$, where $y^\a$ restrict to coordinates on $\p M$.
Our stability estimate will be in terms of certain weighted Sobolev spaces on which 0-pseudodifferential operators naturally act:
we let $dV_g$ be the Riemannian volume density on $M$ induced by $g$ and for $k\in \mathbb{N}_0=\{0,1,\dots\}$ we let
\begin{equation}
	x^{\d}H_0^k(M,dV_g)=\{u\in x^\d L^2(M,dV_g):x^{-\d}V_1
	\cdots V_m u\in L^2(M,dV_g),\:
	 m\leq k, \: V_j\in \calV_0\}.
\end{equation}
If $s\geq 0$ then $H_0^s(M,dV_g)$ is defined by interpolation and for $s<0$ by duality.
Fixing vector fields $V_j\in \calV_0$ in coordinate patches we can make sense of the norms $\|\cdot\|_{x^\d H_0^{k}(M,dV_g)}$. 
As we will show in Section \ref{sec:pseudodifferential_property}, it turns out that $I$ and $\calN_g$ can be extended to operators on  $x^{\d}L^2(M,dV_g)$  for $\d>-n/2$, bounded into appropriate weighted Sobolev spaces; specifically  for $\calN_g$ we have that it is bounded $x^{\d}L^2(M,dV_g)\to x^{\d'}H_0^1(M,dV_g)$ provided $\d'\leq \d$, $\d>-n/2$ and $\d'<n/2$.
The main result of the paper is as follows:
{}
\begin{theorem}\label{thm:main}
	Let $(\cM^{n+1},g)$ be a simple AH manifold, $n\geq 1$.
	 Then $I$ and $\calN_g=I^*I$ are injective on $x^\d L^2(M,dV_g)$, $\d>-n/2$.
	Moreover, one has the stability estimate: 
\begin{equation}
	\|u\|_{x^\d H_0^{s}(M,dV_g)}\leq C\|\calN_g  \;u\|_{x^\d H_0^{s+1}(M,dV_g)}, \quad \d\in (-n/2,n/2),\: s\geq 0.
\end{equation}
\end{theorem}
\noindent Note that $x C^\infty(M)\subset x^\d L^2(M,dV_g)$ provided $\d<1-n/2$, so Theorem \ref{thm:main} includes the injectivity result of \cite{2017arXiv170905053G} on simple AH manifolds as a special case.
However, that result is used in an essential way in the proof, similarly to the way the injectivity of $I$ on simple compact manifolds with boundary was used to derive \eqref{eq:normal_operator_estimate_cpt} in \cite{MR2068966}.

As already mentioned, in the case of hyperbolic space a stability estimate as in Theorem \ref{thm:main} follows immediately from the work  of \cite{MR1104811}; moreover, the inversion of the hyperbolic Radon transform on the two dimensional hyperbolic space has been numerically implemented in a stable manner (\cite{MR1463592}, \cite{MR1812480}).
In the AH setting (under some assumptions) stability of the local X-ray transform follows from work in a forthcoming paper by C. Robin Graham and the author (see \cite[Chapter 1]{EptaminitakisNikolaos2020Gxto}).

We briefly outline the idea of the proof of
 Theorem \ref{thm:main}.
As we show in Section \ref{sec:pseudodifferential_property},  $\calN_g$ is an elliptic pseudodifferential operator in $\Psi_0^{-1,n,n}(M)$ in the large 0-calculus (that is, it is a pseudodifferential operator of order $-1$ whose Schwartz kernel vanishes to order $n$ at the side faces of the 0-stretched product, see Section \ref{ssub:stretched_spaces_and_the_0_calculus}). 
Its \textit{model operator}\footnote{The model operator is typically called the normal operator; however, we use this name to avoid confusion with the normal operator $\calN_g$.} can be identified with $\calN_{h}$, where $h$ is the hyperbolic metric on the Poincar\'e ball; using the explicit inversion formulas for $\calN_{h}$ derived in \cite{MR1104811} and methods developed in \cite{MR916753} and \cite{MR1133743} we construct a left parametrix for $\calN_g$.
In \cite{MR916753} and \cite{MR1133743} parametrices were constructed for elliptic 0 and edge differential operators, whereas here we apply those techniques to construct a parametrix for a pseudodifferential operator.
We mention that Fredholm properties of certain classes of 0-pseudodifferential operators were also studied in \cite{MR1965451}.
The parametrix is used in two ways: firstly, one obtains an estimate
\begin{equation}\label{eq:closed_range_estimate}
	\|u\|_{x^\d H_0^{s}(M,dV_g)}\leq C\left(\|\calN_gu\|_{x^\d H_0^{s+1}(M,dV_g)}+\|Ku\|_{x^\d H_0^{s}(M,dV_g)}
	\right)
\end{equation}
for $ \d\in (-n/2,n/2)$ and $s\geq 0,$
where $K:x^\d H_0^{s}(M,dV_g)\to x^\d H_0^{s}(M,dV_g)$ is a compact  operator.
Next, using the Mellin transform and the parametrix it can be shown that any function $u\in x^\d L^2(M,dV_g)$ in the nullspace of $\calN_g$, where $\d>-n/2,$ is smooth in $\cM$ and has a polyhomogeneous expansion at $\p M$, vanishing there to order at least $n$. 
The author is indebted to Rafe Mazzeo for showing him this argument, which is similar in spirit to the constructions of polyhomogeneous expansions for elements in the nullspace of elliptic edge differential operators in \cite[\sectionsymbol 7]{MR1133743}, also see \cite{MR1087800}. 
In  \cite[Proposition 3.15]{2017arXiv170905053G} it is shown that if $u\in x C^\infty(M)$ lies in the nullspace of $I$ then $u$ vanishes to infinite order at $\p M$, and one checks that the proof also works for $u$ a priori assumed polyhomogeneous and vanishing to order at least 1 at $\p M$.
Since the nullspace of $\calN_g$ agrees with that of $I$, it follows that $u$ is in the nullspace of the latter and polyhomogeneous, hence  it vanishes to infinite order at $\p M$.
Once this has been established, the injectivity argument in \cite{2017arXiv170905053G} using  Pestov identities applies to conclude that $u\equiv 0$.
Finally the injectivity of $\calN_g$ together with \eqref{eq:closed_range_estimate} yields Theorem \ref{thm:main} using a standard functional analysis result.

It would be interesting to explore whether stability still holds in the AH setting when one relaxes the simplicity assumption.
In the compact manifold with boundary setting,  presence of conjugate points in the interior of a compact Riemannian surface causes stability to fail in dimension 2 (\cite{MR2970707}, \cite{MR3339183}), and it is natural to expect an analogous behavior in the AH setting.
However, in dimension 3 and higher, additional geometric assumptions can allow for stability even if there are conjugate points (\cite{Uhlmann2016}, \cite{MR3770848}) so it is likely that analogous results hold on AH manifolds.
It would be especially interesting to investigate whether stability or instability holds in the presence of boundary conjugate points. 
It would also be interesting to study stability in the presence of trapped geodesics; as already mentioned, in the case when the trapped set is hyperbolic for the geodesic flow, injectivity of $I$ on $x C^\infty(M)$ is known by \cite{2017arXiv170905053G} and  stability can be shown on compact manifolds with strictly convex boundary, no conjugate points and hyperbolic trapped set (see \cite{MR3600043}).

The paper is organized as follows: 
in Section \ref{sec:geodesic_flow} we provide some background on the geodesic flow and the X-ray transform on AH manifolds following \cite{2017arXiv170905053G}.
Section \ref{ssub:stretched_spaces_and_the_0_calculus} contains background material on the 0-geometry and 0-calculus that will be needed later.
In Section \ref{sec:pseudodifferential_property} we use the form of the distance function on a simple AH manifold (Proposition \ref{prop:distance_function}) to show that the normal operator $\calN_g$ is an elliptic pseudodifferential operator in the 0-calculus.
In Section \ref{sec:model_operator} we identify the model operator of $\calN_g$, which is invertible, as shown in \cite{MR1104811}.
Finally, in Section \ref{sec:parametrix} we construct a parametrix for $\calN_g$, use it to show boundary regularity for elements in its nullspace,
and prove Theorem \ref{thm:main}.
Throughout the paper we use Einstein notation, with Latin indices running from 0 to $n$ and Greek indices from $1$ to $n$.

\section{Geodesic Flow of AH Manifolds and the X-Ray Transform}

\label{sec:geodesic_flow}

In this section we recall facts related to the geodesic flow and the X-ray transform on AH manifolds, which were analyzed in \cite{2017arXiv170905053G}, and show a lemma which will be used in Section \ref{sec:pseudodifferential_property} to prove a mapping property for the X-ray transform.
In this section $(\cM^{n+1},g)$ is a non-trapping AH manifold, with $\cM$  the interior of a compact manifold with boundary $M$.

Each representative $h$ in the conformal infinity of $(\cM,g)$ determines a boundary defining function $x$ for $\p M$, called \textit{geodesic boundary defining function} associated to $h$, such that $x^2{g}\big|_{T\p M}=h$ and $\|dx\|_{x^2{g}}=1$ near $\pM$.
Then via the flow of its gradient, $x$ induces a product decomposition of a collar neighborhood of $\p M$ as $[0,\e)_x\times\p M$, in terms of which the metric is written  near $\p M$ in \textit{normal form} 
 	$$g=\frac{dx^2+h_x}{x^2},$$
where $h_x$ is a smooth 1-parameter family of metrics on $\p M$ satisfying $h_0=h$.
Choosing coordinates $y^\a$ for $\p M$ near a boundary point we can write $g=\frac{dx^2+(h_x)_{\a\b}dy^\a dy^\b}{x^2}$.

Parametrizing the space of geodesics on $\cM$ is more involved than, e.g. on a compact manifold with boundary, due to their behavior near $\p M$ and the fact that they have infinite length.
As shown in \cite{2017arXiv170905053G}, it can be done by introducing an appropriate extension of the unit cosphere bundle $S^*\cM=\{(z,\xi)\in T^*\cM:|\xi|_g=1\}$ down to $\p M$.
Recall that Melrose's b-cotangent bundle ${}^b T^*M$ (see \cite{MR1348401}) is a smooth bundle over $M$ with natural projection $\pi$, canonically isomorphic with $T^*\cM$ over $\cM$ and trivialized locally near the boundary by $({dx}/{x}  , dy^1,\dots,d  y^n)$.
Viewed as a subset of $({}^bT^*M)^\circ$, $S^*\cM$ can be written near $\p M$ as $\{(z,\xi=\oz \frac{dx}{x}+\eta_\a dy^\a)\in ({}^bT^*M)^\circ:\oz^2+x^2|\eta|_{h_x}^2=1\}$.
Thus the closure of ${S^*\cM}$ in ${}^bT^*M$ is a smooth embedded non-compact submanifold of ${}^bT^*M$ with disconnected boundary; we denote it by $S^*M$.
The  Hamiltonian vector field $X$ on $S^*\cM$ associated with the metric Lagrangian $\calL_g=|\xi|_g^2/2$ can be written as 
$X=x \o{X}$, where $\o{X}$ extends to be smooth on ${S^*M}$ and transversal to its boundary: in coordinates it takes the form
\begin{align}\label{eq:Hamiltonian_xbar}
	\oX=\o{\z}\p_x+xh_x^{\a\b}\eta_\a\p_{y^\b}-\big(x|\eta|^2_{h_x}+\frac{1}{2}x^2\p_x |\eta|^2_{h_x}\big)\p_{\o{\z}}-\frac{1}{2}x\p_{y^\a}|\eta|^2_{h_x}\p_{\eta_\a}.
\end{align}
The flow of $\oX$ is incomplete and, since $X$ and $\oX$ are related by multiplication by a scalar function, their integrals curves in $S^*\cM$ agree up to reparametrization.
Orbits of the flow of $\o{X}$ can be parametrized by their ``incoming'' covector, that is, each orbit can be identified with its intersection with the connected component of $\p{S^*M}$ on which $\o{X}$ is inward pointing.
This component, $\p_-{S^*M}:=\{(z,\xi=\oz \frac{dx}{x}+\eta_\a dy^\a)\in {}^bT^*M:x=0,\, \oz=1\}$ is often referred to as the incoming boundary.
 The definition of the outcoming boundary $\p_+S^*M$ is analogous, except $\oz=-1$ there.
 Both of those sets are invariant subsets of ${}^bT^*M\big|_{\p M}$, independent of the choice of coordinates and of $g$.
 Given a choice of conformal representative (which induces a geodesic boundary defining function), $\p_\pm S^*M$ can be identified with $T^*\pM$ via $\mp x^{-1}dx+\eta^\a dy^\a\leftrightarrow \eta_\a dy^\a$.

The unit cosphere bundle $S^*\cM$ has a natural measure $d\l$ called the Liouville measure, induced by the restriction to $S^*\cM$ of the $2n+1$ form $\l=\a\wedge (d\a)^n$, with $\a$ the tautological 1-form on $T^*\cM$.
This measure decomposes as $d\l=dV_g d\mu_g$, where $d\mu_g$ is the measure induced by $g$ on each fiber of $S^*\cM$ and $dV_g$ is the Riemannian volume density on $\cM$.
As shown in \cite[Lemma 2.2]{2017arXiv170905053G}, $xd\l$ extends from $S^*\cM$ to a smooth measure on $S^*M$. 
Moreover, $\iota_X \l$ extends to a smooth $2n$-form on $S^*M$, which restricts to a volume form on $\p_-S^*M$; the latter agrees with the canonical volume form on $T^*\p M$ (induced by the symplectic form there) under the identification described above. 
We will denote the corresponding measure on $\p_- S^*M$ by $d\l_\p$.

Now let $f\in C_c^\infty (S^*\cM)$ and $\phi_t$ be the flow of the Hamiltonian vector field $X$ on $S^*\cM$, which is complete. We define the X-ray transform
 \begin{equation}\label{eq:defn_xray}
 	If(z,\xi)=\int_{-\infty}^\infty f(\phi_t(z,\xi))dt\in C_X^\infty(S^*\cM),
 \end{equation}
where the space $C_X^\infty(S^*\cM)$ consists of smooth functions on $S^*\cM$ constant along the orbits of $X$.
Since $C_c^\infty(\cM)$ can be naturally viewed as a subset
of $C_c^\infty(S^*\cM)$ via pullback, \eqref{eq:defn_xray} reduces to the usual X-ray transform on $C_c^\infty(\cM)$, viewed as an element of $ C_X^\infty(S^*\cM)$.
Now as we mentioned before the vector field $\oX=x^{-1}X$ extends to be smooth on $S^*M$ and transverse to $\p S^*M$.
This implies that any $u\in C_X^\infty(S^*\cM)$ extends smoothly to $S^*M$ down to $\p_\pm S^*M$:
by transversality, the flow of $\oX$ running forward and backward can be used to identify a neighborhood of any point in $\p_\mp S^*M$ respectively with a subset of $[0,\e)_t\times \p_\mp S^*M$; then in terms of this decomposition $u$ is independent of $t$ and thus extends smoothly down to $t=0$.
Therefore the restriction $u\big|_{\p _\pm S^*M}\in C^\infty(\p_\pm S^*M)$ is well defined; conversely, any function in $C^\infty(\p_\pm S^*M)$ can be extended off of $\p_\pm S^*M$ to be constant along the orbits of $\oX$ in $S^*M$, and hence also those of $X$ in $S^*\cM$, thus yielding an element of $C_X^\infty(S^*\cM)$.
This discussion implies that we have an isomorphism 
\begin{equation}\label{eq:identification}
	C_X^\infty(S^*\cM)\rightarrow C^\infty(\p_-S^*M)
\end{equation}
and both spaces are also isomorphic to $C_{\oX}^{\infty}(S^*M)=C^\infty(S^*M)\cap \ker \oX$.
Due to these facts, \eqref{eq:defn_xray} can also be regarded as an element of $C_{\oX}^\infty(S^*M)$, and of $C^\infty(\p_-S^*M)$ upon restricting.
The range of $I$ is actually smaller than $C_{\oX}^{\infty}(S^*M)$ whenever acting on $C_c^\infty(S^*\cM)$ (or $C_c^\infty(\cM)$):
by the discussion on short geodesics in \cite[Section 2.2]{2017arXiv170905053G}, given any compact set $K\subset S\cM$ there exists a compact set $K'\subset \p_-S^*M$ such that any integral curve of $\oX$ starting at $(z,\xi)\not\in K'$ does not intersect $K$.
Moreover, given a compact $K'\subset\p_-S^*M$, the union of all  integral curves of $\oX$ starting at $K'$ forms a compact subset of $S^*M$.
This implies that $I:C_c^\infty(S^*\cM)\to C_{c,\oX}^\infty(S^*M)$, where $C_{c,\oX}^\infty(S^*M)=C_{c}^\infty(S^*M)\cap \ker (\oX)$.

The X-ray transform can be expressed using the non-complete flow $\o{\phi}_\t$ of $\oX$.
As already mentioned, $\o{\phi}_\t$ is   a reparametrization of the flow $\phi_t$ of $X$ in $S^*\cM$:
for $(z,\xi) \in S^*\cM$ one has $\o{\phi}_\t(z,\xi)=\phi_t(z,\xi)$ with $t(\t,(z,\xi))=\int_0^\t\frac{d\s}{x\circ \o{\phi}_
\s(z,\xi)}$.
Moreover, for each $(z,\xi)\in S^*M$ there exist finite $\t_\pm(z,\xi) \geq0$ such that $\o{\phi}_{\pm\t_\pm(z,\xi)}(z,\xi)\in \p_\pm S^*M$.
Thus for $f\in C^\infty_c(S^*\cM)$ (or $f\in C^\infty_c(\cM)$) \eqref{eq:defn_xray} can be rewritten as
\begin{equation}\label{eq:xray_finite_interval}
If(z,\xi)=\int_0^{\t_+(z,\xi)}f(\o{\phi}_\t(z,\xi))\frac{d\t}{x\circ\o{\phi}_\t(z,\xi)}\in C_{c,\oX}^\infty(S^*M). 
\end{equation}

One can identify a formal adjoint $I^*$ of $I$ on appropriate function spaces using suitably chosen inner products.
By \cite[Lemma 3.6]{2017arXiv170905053G}, there is an analog of Santal\'o's formula:
\begin{equation}\label{eq:santalo}
	\int_{S^*\cM} f\, d\l=\int_{\p_-S^*M}{I}f \,d\l_\p ,\quad  f\in C_c^\infty(S^*\cM).
\end{equation}
Note that this implies that $I$ also extends continuously as an operator ${I}: L^1(S^*\cM;d\l)\to L^1(\p_-S^*M;d\l_\p)$ (where the isomorphism \eqref{eq:identification} is used implicitly).
We define an inner product on $C_{c,\oX}^\infty(S^*M)$: for $u_1,u_2\in C_{c,\oX}^\infty(S^*M)$ let
\begin{equation}
	\lg u_1, u_2\rg_\p:=\int_{\p_-S^*M} u_1 \, \o{u_2}\, d\l_\p,
\end{equation}
where on the right hand side $u_1$, $u_2$ are restricted to $\p_-S^*M$; we will generally not write this restriction explicitly.
Now consider the X-ray transform viewed as an operator $I:C_c^\infty(M)\to C_{c,\oX}^\infty(S^*M)$ and define the operator
 $I^*:C_{c,\oX}^\infty(S^*M)\to C^\infty(\cM)$ by
\begin{equation}
	I^*u(z)=\int_{S^*_z\cM}u\, d\mu_g,\quad u\in C_{c,\oX}^\infty(S^*M).
\end{equation}
Considering real valued functions $f\in C_c^\infty(\cM)$ and $u\in C_{c,\oX}^\infty(S^*M)$,  we use \eqref{eq:santalo} to compute
\begin{equation}
\begin{aligned}
  \lg u,If\rg_\p=\int_{\p_-S^*M}&u\, {If}\, d\l_\p=\int_{\p_-S^*M} {I(uf)}\, d\l_\p\\*
  =&\int_{S^*\cM} u\, {f}d\l=\int_{\cM} \Big(\int_{S^*_z\cM}u\, d\mu_g \Big){f}(z)dV_g(z)=\lg I^*u , f\rg_{L^2(M,dV_g)}.
\end{aligned}\label{eq:adjoint}  
\end{equation}
This computation implies that with the stated inner products and function spaces $I^*$ is a formal adjoint for $I$.

We will later need to consider the X-ray transform and the normal operator $\calN_g=I^*I$ acting on functions that live in weighted $L^2$ spaces.
The target space of $I$ will also have to be an appropriately weighted $L^2$ space and as will become apparent soon it is more natural for this discussion to view $If$ as a function on $\p_-S^*M$.
Restriction to $\p_-S^*M$ induces an isometry between $C^\infty_{c,\oX}(S^*M)$ and $C_c^\infty(\p_-S^*M)$ with respect to the the inner product $\lg \cdot, \cdot\rg_\p$ and the $L^2(\p_-S^*M;d\l_\p)$ inner product respectively, so  \eqref{eq:adjoint} can also be rewritten as
\begin{equation}\label{eq:adjoint_2}
 	\lg u,If\rg_{L^2(\p_-S^*M;d\l_\p)}=\lg I^*u , f\rg_{L^2(M,dV_g)}, \quad f\in C_c^\infty(\cM),\quad u\in C_{c,\oX}^\infty(S^*M).
 \end{equation}
By \cite[Lemma 3.8]{2017arXiv170905053G} $I$ extends to a bounded operator ${I}:|\log x|^{-\b} L^2(S^*\cM;d\l)\to L^2(\p_- S^*M;d\l_\p)$ provided $\b>1/2$. This also implies that ${I}:|\log x|^{-\b} L^2(M,dV_g)\to L^2(\p_- S^*M;d\l_\p)$ is bounded.
Hence \eqref{eq:adjoint_2} implies that  $I^*$ extends to a bounded operator $I^*:L^2(\p_- S^*M;d\l_\p)\to |\log x|^{\b} L^2(M,dV_g)$ for $\b>1/2$ (where the action of $I^*$ on a function $u\in L^2(\p_- S^*M;d\l_\p)$ is understood as an action on the extension of $u$ to $S^*M$ so that it is constant along the orbits of $\oX$).
Thus \eqref{eq:adjoint_2} is valid for $u\in L^2(\p_- S^*M;d\l_\p)$ and $f\in |\log x|^{-\b} L^2(M,dV_g)$.
Moreover, 
 the normal operator is bounded
\begin{equation}\label{normal_operator_first_boundedness}
	\calN_g=I^*I:|\log x|^{-\b} L^2(\cM,dV_g)\to|\log x|^{\b} L^2(\cM,dV_g), \quad \b>1/2.
\end{equation}
Using the microlocal properties of $\calN_g$ that we  prove in Section \ref{sec:pseudodifferential_property} and Lemma \ref{lm:sobolev_restriction} below, we will obtain extensions of $I$ and $I^*I$ to larger spaces of functions, in Corollary \ref{cor:boundedness_I}.

The following lemma relates a weighted $L^2$ norm of functions in $C_{\oX}^\infty (S^*M)$ with a weighted $L^2$ norm of their restriction to $\p_-S^*M$.
We set $\lg\eta\rg_h:=\sqrt{1+|\eta|^2_{h}}$.

\begin{lemma}\label{lm:sobolev_restriction}
Let $\d<0$. Then there exists a $C=C_\d>0$ such that if $u\in C_{\oX}^\infty (S^*M)\cap {x^{\d}L^2(S^*\cM;d\l)}$  one has, using the isomorphism  \eqref{eq:identification},
\begin{equation}
	 \frac{1}{C}\|u\|_{\lg\eta\rg_h^{-\d}L^2(\p_-S^*M;d\l_\p)}\leq
\|u\|_{x^{\d}L^2(S^*\cM;d\l)} \leq 
C\|u\|_{\lg\eta\rg_h^{-\d}L^2(\p_-S^*M;d\l_\p)}<\infty.
\end{equation}

\end{lemma}
\begin{proof}
First note that $C_{\oX}^\infty (S^*M)\cap {x^{\d}L^2(S^*\cM;d\l)}\neq \emptyset$ for $\d<0$: 
indeed, let $f\in C_c^\infty(\cM)$, implying that $u=I{f}\in C_{c,\oX}^\infty(S^*M)\subset C_{\oX}^\infty (S^*M)$.
Since $xd\l$ is a smooth measure on $S^*M$, one sees that $x^{-\d}C_{\oX}^\infty(S^*M)\subset L^2_{loc}(S^*M;d\l)$, implying the claim.
Now by \eqref{eq:santalo}, since $u\in C_{\oX}^\infty (S^*M)$, we have
\begin{align}
	\|u\|_{x^{\d}L^2(S^*\cM;d\l)}^2=\int_{S^*\cM}&|x^{-\d} u|^2d\l
{=}\int_{\p_-S^*M}{I}(|u|^2x^{-2\d})d\l_\p\\*
&=\int_{\p_-S^*M}\big|\big({I}(x^{-2\d})\big)^{1/2}u\big|^2 d\l_\p.\label{eq:weighted_x_ray}
\end{align}
The second equality is valid because $|x^{-\d}u|^2\in L^1(S^*\cM;d\l)$.
Let  ${{(z,\xi)}}=\big((0,y),\frac{dx}{x}+\eta_\a dy^\a\big)\in \p_-S^*M$ with $|\eta|_h>C_0>0$.
If $C_0$ is sufficiently large, then  \cite[Lemma 2.8]{2017arXiv170905053G} implies that $x\circ \o{\phi}_\t({{z,\xi}})=|\eta|_h^{-1}\sin(\a_{{(z,\xi)}}(\t))+O(|\eta|_h^{-2})$, 
where $\a_{{(z,\xi)}}:[0,\t_+({{z,\xi}})]\to [0,\pi]$ is a family of diffeomorphisms depending smoothly on $(z,{{\xi}})\in \p_-S^*M$, with $\p_\t\a_{{(z,\xi)}}(\t)=|\eta|_h+O(1)$, and 
$\t_+(z,{{\xi}})=|\eta|^{-1}\pi+O(|\eta|_h^{-2})$ as $|\eta|_h\to\infty$.
So
\begin{align}
	{I}(x^{-2\d})({{z,\xi}})=&
	\int_0^{\t_+({z,\xi})}x^{-1-2\d}\circ \o{\phi}_\t(z,\xi)d\t=\int_0^{\t_+({z,\xi})}\left(|\eta|_h^{-1}\sin(\a_{{(z,\xi)}}(\t))+O(|\eta|_h^{-2})\right)^{-1-2\d}d\t\\	
=&		\int_0^\pi\left(|\eta|_h^{-1}\sin(s)+O(|\eta|_h^{-2})\right)^{-1-2\d}\frac{ds}{|\eta|_h+O(1)}\\
=&		|\eta|_h^{2\d}\int_0^\pi\left(\sin(s)+O(|\eta|_h^{-1})\right)^{-1-2\d}\frac{ds}{1+O(|\eta|_h^{-1})}.
\end{align}
Since $\int_0^\pi\left(\sin(s)+O(|\eta|_h^{-1})\right)^{-1-2\d}\frac{ds}{1+O(|\eta|_h^{-1})}=a_\d+O(|\eta|_h^{-1})$ with $a_\d>0$ for $\d<0$, we find that ${I}(x^{-2\d})({{(z,\xi)}})=a_\d |\eta|_h^{2\d}+O(|\eta|_h^{-1+2\d})$ as $|\eta|_h\to\infty$.
On the other hand, if $|\eta|_{h}\leq C_0$, ${I}(x^{-2\d})$ is uniformly bounded above and below by positive constants depending on $\d$ and $C_0$.
Thus \eqref{eq:weighted_x_ray} is comparable to $\|\lg \eta\rg_h^{\d}u\|_{L^2(\p_-S^*M;d\l_\p)}=\|u\|_{\lg \eta\rg_h^{-\d}L^2(\p_-S^*M;d\l_\p)}$.
\end{proof}

\section{The 0-Geometry and 0-Pseudodifferential Calculus }
\label{ssub:stretched_spaces_and_the_0_calculus}

In this section we recall some background on the 0-geometry and the 0-calculus, which will be used throughout the paper.
The main sources are \cite{MR2941112}, \cite{MR916753} and \cite{MR1133743}, also see \cite{MR1111745}. 
Throughout the section, $M^{n+1}$ will be a compact manifold with boundary and  $(x,y^1,\dots, y^n)$ are  coordinates near a boundary point with $x$ a boundary defining function.

\subsection{The b- and 0-Tangent and Cotangent Bundles, Half Densities}

We already introduced the b-cotangent bundle $^bT^*M$ of $M$ in Section \ref{sec:geodesic_flow}. It is the dual bundle of $^bTM$, which is the bundle over $M$ whose local sections are smooth vector fields tangent to $\p M$ (denoted by $\calV_b$) and which is trivialized by $x\p_{x},\p_{y^{1}},\dots, \p_{y^n}$ locally near $\p M$. We denote by $\Omega^{1/2}_b(M)$ the bundle over $M$ with sections of the form $x^{-1/2}\nu$, $\nu\in C^\infty(M;\Omega^{1/2})$ (here $\Omega^{1/2}$ is the smooth half density bundle).
The 0-tangent bundle $^{0}TM$ is the bundle over $M$ whose local sections are smooth vector fields on $M$ vanishing at $\p M$ (denoted by $\calV_0$ as mentioned in the Introduction); it is trivialized locally near $\p M$ by $x\p_x,x\p_{y^1}\dots,x\p_{y^n} $.
Its dual bundle, ${}^0T^*M$, is trivialized locally near $\p M$ by $dx/x,dy^1/x,\dots ,dy^n/x$.
We let $\Omega_0^{1/2}(M)$ be the smooth complex line bundle over $M$ whose smooth local sections are of the form $x^{-(n+1)/2}\nu$, $\nu\in C^\infty(M;\Omega^{1/2})$.
It is trivialized near $\p M$ by  $x^{-(n+1)/2}|dxdy^1\dots dy^n|^{1/2}$ and in case $M$ is the compactification of an AH manifold $(\cM^{n+1},g)$ then $\Omega_0^{1/2}(M)$  is the geometric half density bundle, globally trivialized by $dV_g^{1/2}$.
We will also occasionally use the notation $\Omega_0^{1/2}(X)$ for $X$ a manifold with corners related to $M$ (such as $X=M^2$ or $X=M^2_0$, the 0-stretched product introduced later); in this case $\Omega_0^{1/2}(X)$ is trivialized by $\prod_j x_j^{-(n+1)/2}\nu$, $\nu\in C^\infty(X;\Omega^{1/2})$, with $x_j$ defining functions for the boundary faces of $X$.
For a manifold with corners $X$ and for ${\star}\in \{\emptyset, b,0\}$ we will write $\dot{C}^\infty(X;\Omega_{\star}^{{1}/{2}})$ for smooth sections of $\Omega_{\star}^{1/2}(X)$ whose derivatives of all orders vanish at $\p X$ and $ C^{-\infty}(X;\Omega_{\star}^{{1}/{2}}):=(\dot{C}^\infty(X;\O_{\star}^{{1}/{2}}))'$.

\subsection{Polyhomogeneous Conormal Distributions}
We will make use of spaces of functions admitting asymptotic expansions at the boundary.
 Let $E\subset \C\times \mathbb{N}_0 $ be an \textit{index set}, that is, a discrete set with the properties
 \begin{align}
 &|(s_j,p_j)|\to\infty\implies\Re(s_j)\to\infty\text{ and }\label{eq:first_assumption}\\
 \label{eq:index_sets_assumptions}
	(s_j,p_j)\in E\implies& (s_j+m,p_j-\ell)\in E, \quad m\in \mathbb{N}_0=\{0,1,\dots\}, \quad \ell=0,1,\dots,p_j.
\end{align}
If a discrete $E\subset \C\times \mathbb{N}_0$ satisfies \eqref{eq:first_assumption} we  write $\o{E}$ to denote the smallest index set containing $E$.
Now let
 $u\in {C}^{-\infty}(M)$; $u$ is \textit{polyhomogeneous conormal with index set} $E$ if it admits an asymptotic expansion in a collar neighborhood $[0,\e)_x\times \p M$ of the boundary of the form 
\begin{equation}\label{eq:polyhomogeneous}
	u\sim \sum_{(s_j,p_j)\in E}\sum_{k=0}^{p_j} x^{s_j}|\log x|^k a_{jk}(y), \quad a_{jk}\in C^\infty( \p M).%,\; (s_j,p_j)\in E.
\end{equation}
If $u$ satisfies \eqref{eq:polyhomogeneous} we write $u\in \calA_{phg}^E$.
By \eqref{eq:index_sets_assumptions}, the property $u\in \calA_{phg}^E$ does not depend on the product decomposition chosen near $\p M$.
Note that the space $\calA_{phg}^E$ is invariant under differentiation by vector fields in $\calV_b(M)$ and that if  $E_1\subset E_2$ 
then $\calA_{phg}^{E_1}\subset \calA_{phg}^{E_2}$.

If $X$ is a manifold with corners with boundary hypersurfaces $X_j$, $j=1,\dots,J$, denote by $\calE=(E_1,\dots,E_J)$ a $J$-tuple of of index sets. The space of polyhomogeneous distributions $\calA_{phg}^\calE(X)$ is defined to be those which have the form \eqref{eq:polyhomogeneous} with $E$ replaced by $E_j$ near the interior of the boundary hypersurface $X_j$ for $j=1,\dots, J$ and which have product type expansions at the intersections of boundary hypersurfaces (for a more rigorous definition see \cite{MR1133743}).
We now list a few shorthand notations.
If $E$ is an index set we will write $E+\ell=\{(s+\ell,p):(s,p)\in E\}$.
The notation $\Re(E)> C$ will mean $\Re(s)>C $ for all $(s,p)\in E$ and $\Re(E)\geq C$ will mean that either $\Re(E)>C$, or $\Re(s)\geq C$ for all $(s,p)\in E$ and $E\cap (\{\Re{z}=C\}\times \{1,2,\dots\})=\emptyset$.
Thus $\Re(E)\geq 0$ implies that $u\in \calA_{phg}^E$ is bounded.
If it is known that $E\subset \R\times \mathbb{N}_0$ we will often write $E\geq C$ or $E>C$.
Whenever $u\in\calA_{phg}^\calE(X)$ is smooth down to a boundary hypersurface $X_j$ and vanishing to order $k$ there we will be replacing $E_j$ in $\calE$ by $k$: in this case $E_j\subset \mathbb{N}_0\times \{0\} $.

If $E$ is a vector bundle over $X$ the discussion above can be used to define  polyhomogeneous conormal sections, written as 
$\calA_{phg}^{\calE}(X;E)$.

\subsection{The Stretched Product}
\label{ssub:stretched}
Here we outline the construction of the 0-stretched product, which is a special case of a blow-up.
For a detailed exposition see \cite{MR2941112}, \cite{MR916753}; more generally for the blow-up construction see \cite{daomwk}.

If $M^{n+1}$ is a compact manifold with boundary, then the \textit{0-stretched product} $M_0^2:=[{M}^2;\p\Di]$ is by definition the space obtained by blowing up
the boundary of the diagonal  $\Di=\{(z,z):z\in {M}\}$ (see Figure \ref{fig:0_product}).
As a set, $M_0^2= (M^2\setminus \p\Di)\bigsqcup SN^{++}(\p \Di)$, where $SN^{++}(\p \Di)$ is the inward pointing spherical normal bundle of $\p \Di$ with fiber $SN^{++}_{(p,p)}\p\Di=\big(\big((\o{T_{p}^+M})^2/T_{(p,p)}\p\Di\big)\setminus 0\big)/\R^+$ at $(p,p)\in \p \Di$, where $T_{p}^+M$ is the inward pointing tangent space. $M_0^2$ is endowed with a natural smooth structure making it into a manifold with corners of codimension up to 3 such that the blow down map $\b_0:M_0^2\to M^2$, $\b_0\big|_{(M^2\setminus \p\Di)}=id$, $\b_0\big|_{SN^{++}_{(p,p)}\p\Di}=(p,p)$ is smooth. 
Under $\b_0$, smooth vector fields on $M^2$ tangent to $\p \Di$ lift to be smooth and tangent to the boundary faces of $M_0^2$.
The boundary face $SN^{++}(\p\Di)\subset {M_0^2}$ is called the \textit{front face} and denoted by $\ff$; we also let $\Di_0=\o{\b_0^{-1}(\Di\setminus\p \Di)}$.
The side faces are $\lf:=\o{\b_0^{-1}(\p M\times M)}$ and  $\rf:=
\o{\b_0^{-1}(M\times\p M)}$.
We will use the notation $x_{\ell}, x_r, x_f$ to refer to a defining function for $\lf$, $\rf$, $\ff$ respectively.
Moreover,  $\calA_{phg}^\calE(M_0^2)$ with $\calE=(E_\ell,E_r,E_f)$ will denote polyhomogeneous distributions on $M_0^2$ with $E_{\ell}, E_r, E_f$ index sets corresponding to  $\lf$, $\rf$, $\ff$ respectively.
Throughout the paper we will make use of the following projective coordinate systems on $M_0^2$:
if $(x,y) $ is a coordinate system near $p\in \p M$ with $x$ a boundary defining function and $(\td{x},\td{y})$ a copy of it on the right factor of $M^2$, the coordinate system
\begin{align}
  \label{eq:coord_lf}&(\td{x},\td{y},s=x/\td{x},{W}=(y-\td{y})/\td{x})\text{  is valid near }\lf\cap \ff\text{ and away from }\rf \text{, whereas}\\
   \label{eq:coord_rf}&(x,y,t=\td{x}/x,Y=(\td{y}-y)/x)\text{  is valid near }\rf\cap \ff\text{ and away from }\lf.
\end{align}
In terms of \eqref{eq:coord_lf} (resp. \eqref{eq:coord_rf}) $s$ (resp. $t$) is a defining function for $\lf$ (resp. $\rf$) and $\td{x}$ (resp. $x$) is a defining function for $\ff$; moreover,
$(s,W)$ (resp. $(t,Y)$) restrict to coordinates in the interior of the front face, smooth down to $(\ff\cap \lf)^\circ$ (resp. $(\ff\cap \rf)^\circ$).
\begin{figure}[h]
\begin{center}
	\includegraphics[scale=.4]{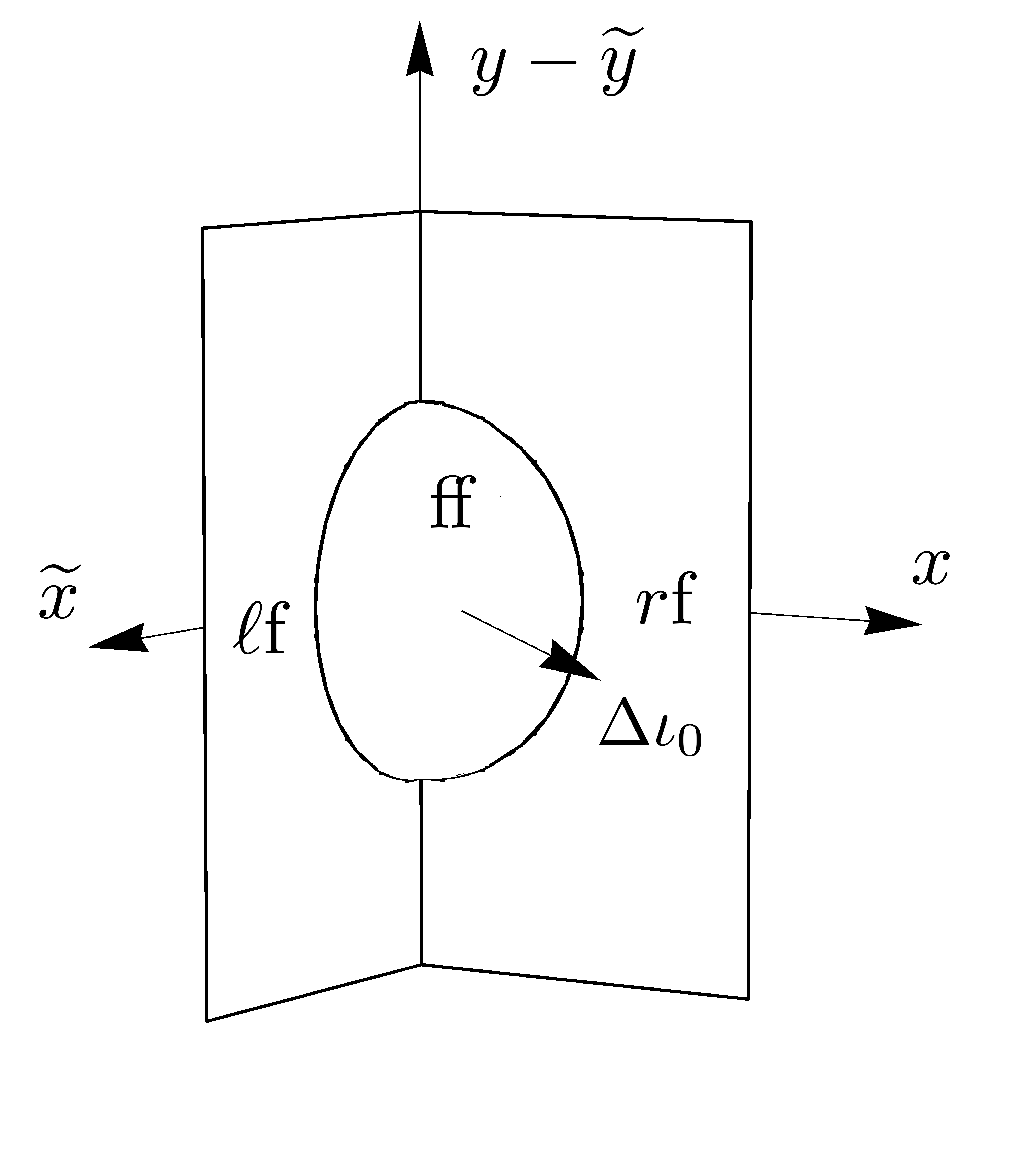}
	\caption{The 0-stretched product.}
	\label{fig:0_product}
	\end{center}
\end{figure}

The fibers of the front face carry additional structure: fix $p\in \p M$ and let ${{T_p^+M} }=\{v\in T_pM:dx(v)> 0\}$.
The subgroup $G_p $ of $GL(T_pM)$ that preserves $\o{{{T_p^+M} }}$ and fixes
 $\p (T_p^+M)$ pointwise induces an invariantly defined free and transitive action on ${T_p^+M} $.
Using linear coordinates $(u,w)=(u,w^1,\dots,w^n)$ on $T_pM$ induced by coordinates $(x,y)$ near $p$,
the action of $(a,b)\in G_p\cong \R^+\ltimes \R^n$ is given by
 $(a,b)\cdot(u,w) =(au,w+ub)$
and the group multiplication in $G_p$ is given by $(a,b)\cdot (a',b')=(aa',b'+a'b)$.
The actions of $G_p^\ell:=G_p\times Id$ and $G_p^r:=Id\times G_p$ on $ \big({T_p^+M} \big)^2$ descend to the interior of the fiber $SN^{++}_{(p,p)}\p\Di=\ff_p$, defining transitive and free actions there.
Moreover, each fiber of the front face has a canonically defined singled out point $e_p$, given by  $\p \Di_0\big|_{(p,p)}$.
Thus one obtains diffeomorphic identifications of $\ff_p^\circ$ with $G_p^\ell\cong G_p^r\cong G_p$ and $\ff_p^\circ$ has two group structures, both canonically isomorphic to $G_p$.
The diffeomorphisms $f_p^\ell, f_p^r:\ff_p^\circ\to G_p$ obtained this way are given (using linear coordinates $(u,w),(\td{u},\td{w})$ on $(T_p^+M)^2$) as $f_p^\ell\big([(u,w),(\td{u},\td{w})]\big)=\big(u/\td{u},(w-\td{w})/\td{u}\big)$, $f_p^r\big([(u,w),(\td{u},\td{w})]\big)=\big(\td{u}/u,(\td{w}-w)/{u}\big)$.
Those diffeomorphisms have equivariance properties: for $q\in G_p$ write $q_l=(q, id)\in G_p^\ell$ and $q_r=(id, q)\in G_p^r$ to obtain as in \cite[\sectionsymbol 3]{MR916753}
that for $\omega \in \ff_p^\circ$ one has 
\begin{equation}
\begin{aligned}\label{eq:equivariance} 
	f_p^\ell(q_l\cdot \omega)=q\cdot f_p^\ell(\omega),& \quad f_p^r(q_r\cdot \omega)=q\cdot f_p^r(\omega)\\
f_p^\ell(q_r\cdot \omega)=f_p^\ell(\omega)\cdot q^{-1}, & \quad f_p^r(q_l\cdot \omega)=f_p^r(\omega)\cdot q^{-1}.
\end{aligned}
\end{equation}

If  $g$ is an AH metric on $\cM$ and $p\in \p M$ one obtains a canonical hyperbolic metric $h_p$ of curvature $-1$ on $T_p^+M$.
If $x$ is a boundary defining function for $\pM$ it is given  by 
\begin{equation}\label{eq:hyperbolic_metric}
	h_p\big|_v:=(dx(v))^{-2}\o{g}\big|_p,\quad v\in {T_p^+M}, 
\end{equation}
where $\o{g}=x^2 g$ and the inner product $\o{g}\big|_p$ on $T_pM$ is naturally identified with an inner product on $T_v(T_pM)$ for any $v\in T_p^+M$.
One checks that \eqref{eq:hyperbolic_metric} does not depend on the choice of the boundary defining function $x$, and with an appropriate choice of coordinates $(x,y)$ near $p$ one can always arrange that $h_p=u^{-2}({du^2+|dw|^2})$ in terms of induced linear coordinates $(u,w)$ on $T^+_pM$.
The metric $h_p$ can be appropriately pulled back to $\ff_p^\circ$ in two ways:
since the action of $G_p$ on $T_p^+M$ is free and transitive, given $v\in T_p^+M$ one can define a diffeomorphism 
\begin{equation}
  f_p^v:G_p\to T_p^+M, \quad G_p\ni q\mapsto q\cdot v\in T_p^+M.
\end{equation}
Thus for each $v$ one obtains a hyperbolic metric $(f_p^v)^*h_p$ on $G_p$ which is right invariant with respect to the group structure of $G_p$, as can be checked in coordinates.
Using this right invariance and the fact that for any two $v,v'\in T_p^+M$ there exists $\td{q}\in G_p$ such that $v'=\td{q}\cdot v$, one checks that the metric $(f_p^v)^*h_p=:h_{G_p}$ is in fact independent of $v$. 
Hence   $h_p^\ell:=(f_p^\ell)^*h_{G_p}$ and $h_p^r:=( f_p^{r})^*h_{G_p}$ are hyperbolic metrics on $\ff_p^\circ$, which are invariant with respect to the left action of $G_p^r$, $G_p^\ell$ respectively, by  \eqref{eq:equivariance}.
Again with an appropriate choice of coordinates we can arrange that
$h_p^\ell=s^{-2}(ds\,^2+|d{W}|^2)$ in terms of coordinates $(s,W)$ on $\ff_p$ as in \eqref{eq:coord_lf}, and similarly that 
 in terms of $(t,Y)$  in \eqref{eq:coord_rf}, $h_p^r=t^{-2}(dt+|dY|^2)$.

\subsection{The 0-Calculus}

We recall the definition and properties of the 0-calculus of pseudodifferential operators that we need later.
As already mentioned, 0-differential operators of order $m\in \mathbb{N}_0$, denoted by $\Diff_0^m(M)$, are those differential operators which can be written as finite sums of at most $m$-fold products of vector fields in $\calV_0$: in coordinates $(x,y) $ near $\p M$, $P\in\Diff_0^m(M)$ can be written as
\begin{equation}\label{eq:zero_differential}
	P=\sum_{j+|{\boldsymbol{\a}}|\leq m} a_{j,{\boldsymbol{\a}}}(x,y)(x\p_x)^j(x\p_y)^{\boldsymbol{\a}}, \quad a_{j,{\boldsymbol{\a}}}\in C^\infty,
\end{equation}
using multi-index notation.

Pseudodifferential operators in the 0-calculus are defined via the lifts  of their Schwartz kernels to $M_0^2$.
As shown in \cite{MR916753}, smooth sections of $\smsec({M}^2)\cong \pi_l^*\smsec(M)\otimes\pi_r^*\smsec(M)$ lift via $\b_0$ to smooth sections of $\smsec(M_0^2)$.\footnote{For any product  space, $\pi_l, \pi_r$ will generically denote projection onto the left and right factor respectively.}
Thus for the Schwartz kernel $\k_P\in C^{-\infty}({M}^2;\smsec)$ of an operator $P:\dot{C}^\infty({M};\O_0^{1/2})\to C^{-\infty}({M};\O_0^{1/2})$ we have $\b_0^*\k_P\in C^{-\infty}(M_0^2;\smsec)$.
The {small 0-calculus} of order $m$, denoted by $\Psi_{0}^m({M})$, consists of operators  whose Schwartz kernel $\k_P$ satisfies $\b_0^*\k_P\in\calA_{phg}^\calE I^m(M_0^2,\Di_0;\smsec)$ with $\calE=\big(\emptyset,\emptyset,\o{\{(0,0)\}}\big)$. 
 This means by definition that $\b_0^*\k_P$ is a section of $\Omega_0^{1/2}(M_0^2)$ conormal of order $m$ to $\Di_0$, smooth down to the front face away from $\Di_0$, and vanishing to infinite order at the side faces.
Recall that a distribution $u$ on a manifold with corners $X^d$ is said to be 
  conormal of order $m\in \R$ with respect to an interior p-submanifold $Y^{d-s}$ (see \cite{MR1133743})  if whenever $Y$ is given locally as the zero set $\{y'=0\}$ in terms of coordinates $(x,y',y'')\in [0,\infty)^k\times \R^s\times \R^{d-k-s}$ for $X^d$,
we have  $u(x,y)=\int_{\R^s} e^{iy'\cdot \xi'}a(x,y'',\xi')d\xi'.$
Here $a\in S^{m'}\big(([0,\infty)^k\times \R^{d-s-k})\times \R^s\big)$ is a symbol  of order $m'=m+d/4-s/2$ (following H\"ormander's convention, see \cite[\sectionsymbol 18.2]{MR2304165}); by definition
 the symbol $a$ satisfies symbol estimates
  $$|\p_x^{\boldsymbol{\a}} \p_{y''}^{\boldsymbol{\b}} \p_{\xi'}^{\boldsymbol{\g}} a(x,y'',\xi')|\leq C_{\boldsymbol{\a},\boldsymbol{\b},\boldsymbol{\g}}\lg\xi'\rg^{{m'-|\boldsymbol{\g}|}},$$
where  $\lg \cdot \rg=(1+|\cdot|^2)^{1/2}$, and ${\boldsymbol{\a}}$, ${\boldsymbol{\b}}$, $\boldsymbol{\g}$
 are multi-indices.
We write  $u\in I^m(X,Y)$ for such conormal distributions; conormal sections of a bundle are defined similarly.
To any operator $P\in \Psi_0^m(M)$ corresponds a principal symbol encoding the leading conormal singularity at $\Di_0$.
This symbol is a symbolic section of a bundle over $N^*\Di_0,$ but as it turns out, using the canonical identification of $N^*\Di_0\leftrightarrow {}^0T^*M$,  it can be canonically identified with a symbol $\s_0^m(P)\in S^{\{m\}}({}^0T^*M):= S^m({}^0T^*M)/S^{m-1}$ (see \cite{MR916753}, \cite{MR1965451} for details), which is called the \textit{principal symbol of $P$}.
Provided there exists a symbol $a\in S^{\{-m\}}({}^0T^*M)$ such that $\sigma_0^m (P)\cdot a\equiv 1$, $P\in \Psi_0^m(M)$ will be called \textit{elliptic}.
Below we set $\Psi_0^{-\infty}({M}):=\bigcap_{m\in \R}\Psi^m_0({M})$.

The operators in the large 0-calculus have kernels with  $\b_0^*\k_P\in\calA_{phg}^\calE I^m(M_0^2,\Di_0;\smsec)$, for  $\calE=(E_\ell,E_r,E_f)$ and $m\in \R$ and are denoted by  $\Psi_0^{m,\calE}(M)$. 
The definition means that $\b_0^*\k_P\big|_{(M_0^2)^\circ}\in I^m((M_0^2)^\circ,\Di_0^\circ;\smsec)$ and it has an asymptotic expansion of the form \eqref{eq:polyhomogeneous} at the boundary  faces with index sets determined by $\calE$, with the coefficients $a_{jk}$ corresponding to the side faces being smooth in their interior and those corresponding to $\ff$ being  conormal of order $m+1/4$ with respect to $\ff\cap \p \Di_0$ (the change in the order of conormality is due to H\"ormander's convention).\footnote{Our definition of the large calculus is consistent with the one given in \cite{MR916753} (except for the fact that we demand that the kernels be polyhomogeneous conormal to the boundary faces of the stretched product and not merely conormal, see \cite{MR916753}) but it differs from the one in \cite{MR1133743}, where the author defines the large edge calculus, of which the 0-calculus is a special case, as $\Psi_e^{m,\calE}=\Psi_e^m+\Psi_e^{-\infty,\calE}$.}
The subspace $\Psi_0^{-\infty,\calE}({M})$ consists of
 operators whose kernels are smooth in $(M_0^2)^\circ$ with polyhomogeneous expansions at the boundary faces.
We will often write $\Psi_0^{m,E_\ell,E_r}({M})$ to imply that $E_f=\o{\{(0,0)\}}$.
In this case, $\Psi_0^{m,E_\ell,E_r}(M)=\Psi_0^m(M)+\Psi_0^{-\infty, E_\ell,E_r}(M).$
 The rest of the shorthand notations for index sets outlined earlier will apply for $\Psi_0^{m,\calE}$; for instance $P\in \Psi_0^{m,a,E_r,E_f}(M)$, $a\in \mathbb{N}_0$, indicates that $\b_0^*\, \k_P$ is smooth near the interior of $\lf$ and vanishes at  $\lf$ to order $a$.

To clarify the action of $P$ in terms of coordinates, 
locally near $\p M $ write $\g_0=|x^{-n-1}dxdy|^{1/2}$ $\in C^\infty(M;\Omega^{1/2}_0)$,
 $\k_P=K_P\cdot \pi_l^*\g_0\otimes \pi_r^*\g_0$ and use the notation $\b_0^*\,K_P^\ell$  for $\b_0^*\,K_P$ expressed in terms of coordinates  $(\td{x},\td{y},s,{W})$ on $M_0^2$.
Then for $P\in\Psi_0^{m,\calE}(M)$ and $f\in \dot{C}^\infty(M)$ we have
	\begin{equation}\label{eq:kernel_coordinates}
		P(f\cdot \g_0)(x,y)=\int \b_0^*\,K_P^\ell\Big(\frac{x}{s},y-\frac{{W}}{s}x,s,{W}\Big)f\Big(\frac{x}{s},y-\frac{{W}}{s}x\Big)
    \frac{|dsd{W}|}{s}\cdot \g_0.
	\end{equation}

Operators in the large 0-calculus can be composed under compatibility assumptions.
The following proposition follows from the proof of \cite[Theorem 3.15]{MR1133743} with a change in normalizations. As stated below it can also be found in \cite{Albin}.

\begin{proposition}\label{prop:composition}
	Let $P\in \Psi_0^{m,\calE}(M)$, $P'\in \Psi_0^{m'\!,\calF}(M)$.
	If $\Re(E_r+F_\ell)>n$ then the composition $P\circ P'$ is defined and $P\circ P'\in \Psi_0^{m+m',\calW}(M)$, where $\calW$ is given by
	\begin{equation}
		\begin{tabular}{c c c}
			$W_\ell=(F_\ell+E_f)\o{\cup}E_\ell$, & $W_{{r}}=(E_r+F_f)\o{\cup}F_{r}$, &$W_f=(E_\ell+F_{r})\o{\cup}(E_f+F_f)$;
		\end{tabular}
	\end{equation}
	here the sum and extended union respectively of the index sets $E$, $E'$ are given by
	\begin{align}
	 	E+E'=&\{(s,p)+(s',p'):(s,p)\in E, \; (s',p')\in E'\},\\
	 	E\o{\cup}E'=&E\cup E'\cup\{(s,p+p'+1):\text{there exist }(s,p)\in E, (s,p')\in E'\}.
	 \end{align}
	 	
\end{proposition}

We also state results regarding mapping properties on polyhomogeneous functions and on half densities in Sobolev spaces.
For a proof of the following, see  \cite{Albin}, \cite{MR1111745}, \cite{MR1133743}:

\begin{proposition}\label{prop:polyhomogeneous_mapping}
	Let $u\in \calA_{phg}^F(M;\O_0^{1/2})$ and $P\in \Psi_0^{m, \calE}(M)$, $m\in \R$. If $\Re(E_r+F)>n$
	then $Pu\in \calA_{phg}^{F'}(M;\O_0^{1/2})$, where $F'=E_\ell\o{\cup}(E_f+F)$.

\end{proposition}

The mapping properties of the large 0-calculus in terms of weighted Sobolev half densities, denoted by $x^{\d}H_0^s(M;\O_0^{1/2})$, will be important later.
We remark that $C_c^\infty (M;\O_0^{1/2})$ (hence also $\dot{C}^\infty (M;\O_0^{1/2})$) is dense in $x^{\d}H_0^s(M;\O_0^{1/2})$ for $s\geq 0$ (see  \cite[Lemma 3.9]{lee2006fredholm}). Also, the inclusion $x^{\d'} H_0^{m'}(M;\O_0^{1/2})\hookrightarrow x^\d H_0^m(M;\O_0^{1/2})$ is compact provided $m'>m$ and $\d'>\d$.
Proposition \ref{prop:boundedness} below follows from the proofs of Corollary 3.23 and Theorem 3.25 in  \cite{MR1133743} upon taking into account the different conventions regarding the definition of operators in the large 0-calculus and the different densities on which they act:

\begin{proposition}\label{prop:boundedness}
Let $P\in \Psi_0^{m,\calE}(M)$, $m\in \R$. Provided $s\in \R$, $\Re(E_r)>n/2-\d$, $\Re(E_f)\geq\d'-\d$ and  $\Re(E_\ell)>\d'+n/2$ 
one has that
\begin{equation}\label{mapping_property}
	P:x^\d H_0^s(M;\O_0^{1/2})\to x^{\d'} H_0^{s-m}(M;\O_0^{1/2})	
\end{equation}
is bounded.
In particular, if $m<0$, $\Re(E_r)>n/2-\d$, $\Re(E_f)>0$ and  $\Re(E_\ell)>\d+n/2$ then $P:x^\d H_0^s(M;\O_0^{1/2})\to x^{\d} H_0^{s}(M;\O_0^{1/2})	$ is compact.
\end{proposition}

\subsection{The Model Operator}\label{ssec:model_operator_background}

Given $p\in \p M$, an operator $P\in\Psi_0^{m,\calE}(M)$ with $\Re(E_f)\geq 0$ determines an invariantly defined operator $N_p(P)$ on $T_p^+M$, which we call the \textit{model operator}\footnote{As already mentioned, the more common name for the model operator is normal operator. Despite not following the usual convention for its name, we maintain the traditional notation $N_p$.};
it is closely related to the group structures on $\ff_p$ and captures the 0-th order behavior of the Schwartz kernel of $P$ at the front face.
Consider  neighborhoods $\calU'\subset T_pM$ and $\calU\subset M$
 of $0$ and $p$ respectively, 
 and diffeomorphism
   $\phi:\calU'\to \calU $ { with }$\phi(0)=p, $ $ d\phi\big|_0=Id,${ and }$\phi(T_p\p M)\subset\p M$.
Also let $R_r:{T_p^+M} \to {T_p^+M} $, $r\in (0,\infty)$, be the canonical radial action.
If $P\in\Diff_0^m(M)$ and $f\in C_c^\infty({T_p^+M} )$ the model operator is  defined by 
\begin{equation}\label{eq:model_operator}
	N_p(P)f=\lim_{r\to 0}R_r^*\phi^*P(\phi^{-1})^*R_{1/r}^* f.
\end{equation}
As shown in \cite{MR916753} $N_p(P)$ is independent of the choice of $\phi$ and
 given by freezing the coefficients of $P$ at $p$.
For $P\in\Psi_0^{m,\calE}(M)$ we still define $N_p(P)$ by \eqref{eq:model_operator}, but with $f\in C_c^\infty({T_p^+M} ;\O_0^{1/2})$ and the limit in the space of distributional half densities.
Here $\Omega_0^{1/2}(\o{T_p^+M})$ is the half density bundle trivialized by a global section of the form $u^{-(n+1)/2}|dudw|^{1/2}$ in linear coordinates as above; henceforth we denote such a section by $\g_p$.
If an AH metric $g$ is chosen on $\cM$, it also determines a hyperbolic metric on $T_p^+M$ and the various density bundles are naturally trivial, so \eqref{eq:model_operator} is also naturally defined for $P\in \Diff_0^m(M)$ and $f$ a half density.

The model operator can be appropriately interpreted as a convolution operator.
As already mentioned, the interior $\ff_p^\circ$ of each fiber   of the front face carries two group structures isomorphic to the group $G_p\subset GL(T_pM)$ and hence $\ff_p^\circ$ acts on $T_p^+M$ from the left in two ways.
For the rest of the section we use the map $f_p^\ell$ to identify $\ff_p$ with $ G_p$ without writing it explicitly.
Given $v\in T_p^+M$, a section $\g_p$ pulls back via the map $\ff_p^\circ \ni q\mapsto f_p^v(q^{-1})\in T_p^+M$ to a left-invariant half density $\g_H$ on $\ff_p$. In coordinates $(s,W)$ on $\ff_p$ (see \eqref{eq:coord_lf}) $\g_H$ is a constant multiple of $|s^{-1}ds dW|^{1/2}$. We denote by ${}^\ell\Omega_H^{1/2}(\ff_p)$ the bundle spanned by $\g_H$ over $C^\infty(\ff_p)$.
Given a distributional half density $\td{u}=(u\cdot \g_H)\otimes \g_p\in (C_c^\infty(\ff_p^\circ;{}^\ell\Omega^{1/2}_H))'\otimes C^\infty\big(\o{T_p^+M};\Omega_0^{1/2}\big)$ 
 one can define an operator on $T_p^+M$ by left convolution,  i.e. by
\begin{equation}\label{eq:convolution_model}
	\td{u}*(f\cdot \g_p)(v):=\int u(q) f(q^{-1}\cdot v)\g_H^2(q)\cdot \g_p,\quad f\cdot\g_p \in C_c^\infty(T_p^+M;\Omega_0^{1/2}).
\end{equation}
Since the lifted kernel of an operator  $P\in \Psi_0^{m,\calE}(M)$ with $\Re(E_f)\geq 0$ is continuous down to the front face with values in  distributional sections of $\Omega_0^{1/2}(M_0^2)$ conormal to  $\Di_0$, and $\Di_0$ is transversal to $\ff_p$, $P$ determines a distributional half density on $\ff_p$ by restriction: 
\begin{equation}\label{eq:restriction}
\b_0^*\k_P\big|_{\ff_p}\in \calA_{phg}^{E_\ell,E_r} I^{m+\frac{n+1}{4}}\big(\ff_p,\{e_p\};(\pi_l^*\smsec(M)\otimes\pi_r^*\smsec(M))_{(p,p)}\big);
\end{equation}
in \eqref{eq:restriction} $E_\ell$ (resp. $E_r$) corresponds to the expansion at $\lf\cap \ff$ (resp. $\rf\cap \ff$).
Note that fiber elements in $\g_0\big|_p\in\smsec(M)\big|_p$ can be identified with constant multiples of $\g_p$
 after pulling back $\g_0$ by $\phi\circ R_r$ and taking a limit as $r\to 0$.
Moreover, the diffeomorphism $\ff_p^\circ\times T_p^+M\ni(q,v)\mapsto (v,q^{-1}\cdot v)\in (T_p^+M)^2$ can be used to pull back $\pi_\ell^*\g_p\otimes \pi_r^*\g_p$ to a constant multiple of $\g_H\otimes \g_p$, 
so under those identifications we let
\begin{equation}
  F_p(P):=\b_0^*\k_P\big|_{{\ff}_p}\in\calA_{phg}^{E_\ell,E_r} I^{m+\frac{n+1}{4}}(\ff_p,\{e_p\};{}^\ell\Omega_H^{1/2})\otimes \text{span}_\C\{\g_p\}\label{eq:restriction2}.
\end{equation}
By \cite[Proposition 5.19]{MR916753}, for operators with smooth kernel down to the interior of the front face there exists a short exact sequence 
\begin{align}
	0\hookrightarrow \Psi_0^{-\infty,E_\ell,E_r,1}(M)\to &\Psi_0^{-\infty,E_\ell,E_r}(M)\overset{F_*}{\to}\calA_{phg}^{E_\ell,E_r}\big(\ff_*;{}^\ell\Omega_H^{1/2})\otimes \text{span}_\C\{\g_*\}\to 0.\label{eq:short_exact}
\end{align}

Unraveling the definitions above and using the coordinate expression \eqref{eq:kernel_coordinates} (assuming that the coordinates $(x,y)$ are centered at $p$) one checks that
\begin{equation}
 N_p(P)(f\cdot\g_p)(u,w)= \int \b_0^* K_P^\ell(0,0,s,{W})f\Big(\frac{u}{s},w-\frac{{W}}{s}u\Big)\frac{|dsd{W}|}{s}\cdot \g_p\label{eq:normal_operator_coordinates}=(F_p(P)*{(f\cdot\g_p)})(u,w),
\end{equation}
which implies the following:
 \begin{lemma}\label{lm:agreement}
 Let $P\in \Psi_0^{m,\calE}(M)$ with $\Re(E_f)\geq 0$. Then for each $p\in \pM$ and $f\cdot \g_p\in C_c^\infty(T_p^+M;\Omega_0^{1/2})$  one has $N_p(P)(f\cdot \g_p)=F_p(P)*(f\cdot \g_p)$.
 \end{lemma}

Under suitable assumptions the model operator is a homomorphism under composition.
The following proposition  is stated and proved in \cite{MR916753} in the case $P\in \Diff_0^m(M)$. 
It is also stated in \cite{Albin} and in \cite{MR1111745} that the homomorphism property holds, with no assumptions explicitly mentioned.
A detailed proof can be found in \cite{EptaminitakisNikolaos2020Gxto}.
{}
\begin{proposition}\label{prop:homomorphism}
	Let $P$, $P'$ be as in Proposition \ref{prop:composition}, 
	with the additional assumptions $\Re(E_f)\geq 0$, $\Re(F_f)\geq 0$ and $\Re(E_\ell+F_r)>0$.
	Then for each $p\in \p M$ one has $N_p(P\circ P')=N_p(P)\circ N_p(P')$. 
\end{proposition}
\begin{remark}
	The assumptions $\Re(E_f)\geq 0$, $\Re(F_f)\geq 0$ and $\Re(E_\ell+F_r)>0$ guarantee that $P$, $P'$ and $P\circ P'$ have well defined model operators  (see Proposition \ref{prop:composition}).
\end{remark}

If $\cM$ is endowed with an AH metric $g$, by writing the hyperbolic metric $h_p$ in \eqref{eq:hyperbolic_metric} as $h_p=u^{-2}(du^2+|dw|^2)$, $({T_p^+M},h_p) $ is identified with $(\H^{n+1}=\{(u,w)\in \R^+\times \R^n\},h)$, the hyperbolic upper half space model.
Via the map $(f^v_p)^{-1}:T_p^+M\to \ff_p^\circ$  (recall that we use $f_p^\ell$ to identify $\ff_p^\circ$ and $G_p$) the integration in \eqref{eq:convolution_model} can be pulled back to $T_p^+M$, and $N_p(P)$ can be regarded as an operator on $\H^{n+1}$ written as
\begin{equation}\label{eq:right_action}
    N_p(P)(f\cdot\g_p)(u,w)=\int_{\H^{n+1}} \b^*_0K_P^\ell\Big(0,0,\frac{u}{\td{u}},\frac{w-\td{w}}{\td{u}}\Big)f(\td{u},\td{w})\frac{|d\td{u}d\td{w}|}{\td{u}^{n+1}}\cdot \g_p.
\end{equation}
In \eqref{eq:right_action} and henceforth, whenever an AH metric $g$ has been chosen on $\cM$ it will be assumed that $\g_0=|\sqrt{\det g}dxdy|^{1/2}$ in coordinates and that $\g_p=|\sqrt{\det h_p}dudw|^{1/2}$. 
Conjugating by the Cayley transform, $N_p(P)$ can be interpreted as an operator on the Poincar\'e ball $(\B^{n+1},\frac{4|dz|^2}{(1-|z|^2)^2})$  and one checks that if $P\in \Psi_0^{m,\calE}(M)$ with $\Re(E_f)\geq 0$ then one has $N_p(P)\in \Psi_0^{m,\calE'}(\o{\B^{n+1}})$ with $\calE'=(E_\ell,E_r,\o{\{(0,0)\}}$; thus the model operator also extends to appropriate weighted Sobolev spaces on $\o{\B^{n+1}}$ according to Proposition \ref{prop:boundedness}.

\section{The Pseudodifferential Property}
\label{sec:pseudodifferential_property}

For the rest of the paper we assume a simple AH manifold $(\cM^{n+1},g)$.
 In this section we show that the normal operator $\calN_g$ is a 0-pseudodifferential operator, namely that $\calN_g\in \Psi_0^{-1,n,n}(M)$. 
As an intermediate step we study the distance function induced by $g$.
Once the pseudodifferential property of $\calN_g$ has been established,  we use it to extend $I$  to larger weighted $L^2$ spaces than those of Section \ref{sec:geodesic_flow}.

By following the proof of  \cite[Proposition 19]{MR3473907} one can show the following technical lemma (a detailed proof also appears in \cite{EptaminitakisNikolaos2020Gxto}).

\begin{lemma}
\label{lm:map_extends}
	Let $(\cM,g)$ be a simple AH manifold.
	The map $
	\Phi:T^* \cM\to  M_0^2$, 
	$\big(z,\xi\,\big)\mapsto \big(z,\exp_z\!\big(\xi^\#\big)\big)$
	 extends smoothly to a map $\tPhi:{}^0T^* M\to  M_0^2\nonumber$, where we are using the canonical identification of ${}^0T^* M\big|_{\cM}=\big({}^0T^* M\big)^\circ$ and $T^*\cM$. 
Here $\#$ raises an index with respect to the metric $g$. 
Moreover, the differential of $\tPhi$ at $(z,0)\in {}^0T^* M\big|_{\p M}$ has full rank.
\end{lemma}

The behavior of the distance function $\r$ on AH manifolds away from the diagonal has been studied by various authors, see for instance \cite{MR3532393}, \cite{MR3473907} and \cite{2017arXiv170905053G}, and also \cite{MR3169792} for small perturbations of hyperbolic metric.
As Proposition \ref{prop:distance_function} below indicates, provided $(\cM,g)$ is simple, the lift of the distance function to $M_0^2$ is smooth away from $\Di_0$ and the side faces, however our analysis of $\calN_g$ will also require smoothness of $\b_0^*\r^2$  in a neighborhood of $\Di_0$, all the way to the front face.
We are not aware of this fact explicitly stated in the literature, so we provide a proof.

\begin{proposition}\label{prop:distance_function}Let $(\cM,g)$ be a simple AH manifold and let $\r:\cM^2\to\R$ be the geodesic distance function. 
There exists $\a\in C^\infty(M_0^2\backslash\Di_0)$ such that 
\begin{equation}
	\b_0^*\r=\a-\log(x_{\ell})-\log(x_r),
\end{equation}
where $x_{\ell}$ and $x_r$ are defining functions for the left and right face of $M_0^2$ respectively.
Moreover, $\b_0^*\r^2$ extends to a smooth function on $M_0^2\backslash(\lf\cup \rf)$. 
\end{proposition}

\begin{proof}
The first statement follows from work in \cite{MR3532393}, \cite{MR3473907} and \cite{2017arXiv170905053G} (see \cite[Remark 7]{2017arXiv170905053G}).
We show the second statement.
Assume without loss of generality that $x_{\ell},\,x_r\equiv 1$ in a neighborhood of $\Di_0$.
Since $\r^2$ is smooth near $\Di\cap \cM^2$ and thus $\b_0^*\r^2$ extends to a function in  $C^\infty(M_0^2\backslash(\p\Di_0\cup \lf\cup \rf))$, it is enough to show that $\b_0^*\r^2$ extends to be smooth in a neighborhood of $\p \Di_0$.
By the Inverse Function Theorem, Lemma \ref{lm:map_extends} implies that $\tPhi$ restricted to a neighborhood of a point $(p,0)\in {}^0T^*M\big|_{\p M}$ is invertible.
The inverse, defined in a neighborhood $U\subset M^2_0$ of $\p\Di_0\big|_{(p,p)}$, is smooth down to the front face.
In $U\cap (M_0^2)^\circ$
\begin{align}
	\b_0^*{\r}^2(z,\td{z})=|\exp_z^{-1}(\td{z}\,)|_g^2=\left|\left(\tPhi^{-1}(z,\td{z})\right)^{\!\#}\right|^2_g=\left|\tPhi^{-1}(z,\td{z})\right|^2_{g^{-1}}\label{distance_computation}
\end{align}
using the identification $(M_0^2)^\circ\leftrightarrow \cM^2$. 
Since $g$ induces a non-degenerate quadratic form on the fibers of ${}^0T^*M$, smooth all the way	to the boundary, \eqref{distance_computation} extends smoothly to $\p\Di_0$.
\end{proof}

The proof of the following lemma, which uses the Gauss Lemma, is contained in \cite{MR2068966}. 
\begin{lemma}\label{lm:kernel_of_Ag}
	Let $(M,g)$ be a simple AH manifold and let $z=(z^0,\dots,z^n)$, $\td{z}=(\td{z}\,^0,\dots,\td{z}\,^n)$ two copies of the same (possibly global) coordinate system in each of the two factors of $\cM^2$.
In the set where $(z,\td{z})$ are valid coordinates for $\cM^2$, the kernel of $\calN_g$, viewed as a section of $\O_0^{1/2}(M^2)$, is given by $K_{\calN_g}(z,\td{z})\cdot  \g_0(z)\otimes \g_0(\td{z}\,)$, where
\begin{equation}\label{eq:kernel}
	K_{\calN_g}(z,\td{z}\,)=\frac{2|\det(\p_{z\td{z}}\r^2/2)|}{\r^n(z,\td{z})\sqrt{\det g(z)}\sqrt{\det g(\td{z})}}.
\end{equation}
Recall that on an AH manifold $\g_0(z):=dV_g^{1/2}=|\sqrt{\det g(z)}dz|^{1/2}$.
\end{lemma}

We now prove the following key proposition:

\begin{proposition}\label{prop:pseudodifferential}
Let $(\cM^{n+1},g)$ be a simple AH manifold.
Then $\calN_g\in \Psi_0^{-1,n,n}(M).$
Moreover, it is elliptic. 
\end{proposition}

\begin{proof}
We examine the Schwartz kernel of $\calN_g$ on $\cM^2$ and on the stretched product $M_0^2$.
As noted in \cite{MR2068966}, \eqref{eq:kernel} implies that on $\cM^2$  the kernel of $\calN_g$ agrees with the kernel of a pseudodifferential operator of order $-1$ with principal symbol $C_n|\xi|_g^{-1}$.
Since smooth sections of $\smsec(M^2)$ lift to smooth sections of $\smsec(M_0^2)$ it suffices to study the behavior of $K_{\calN_g}(z,\td{z}\,)$ in \eqref{eq:kernel} and its pullback to $M_0^2$ as $z$, $\td{z}\,\to \p M$, both away from, and near the diagonal.
Throughout the proof, $z=(x,y),$ $\td{z}=(\tx,\ty)$ are representations in terms of two copies of the same coordinate system in each factor of $\cM^2$ such that  $x$, $\tx$ are boundary defining functions.

First note that by the Gauss Lemma $|\det(\p_{z\td{z}}\r^2/2)|=|\det\big(d_{\td{z}}\exp_z^{-1}(\td{z})^\flat\big)|$, so by simplicity $\det(\p_{z\td{z}}\r^2/2)\neq 0$ on $\cM^2$ and the absolute value can be ignored in the process of examining the smoothness properties of $K_{\calN_g}$ and $\b_0^* K_{\calN_g}$.
Moreover, we have a  simplification of \eqref{eq:kernel} away from the diagonal:
note that $\p^2_{z{\td{z}}}(\r^2/2)=\r\p^2_{z{\td{z}}}\r+\p_z\r\otimes\p_{\td{z}}\r$.
Since for $H\in \R^{d\times d}$ and $u, v\in \R^d$ one has 
	$\det(H+u\otimes v)=\det(H)+(\adj(H)u)\cdot v$ by the matrix determinant lemma,
where $\adj(H)$ is the adjugate matrix of $H$ and $\cdot$ denotes the Euclidean dot product,
we have
\begin{align}
\det\big(\p^2_{z{\td{z}}}(\r^2/2)\big)=\r^{n+1}\det(\p_{z\td{z}}^2\r)+\r^n\left(\adj(\p_{z\td{z}}^2\r)\p_z\r\right)\cdot\p_{\td{z}}\r.\label{eq:determinant}
\end{align}
Observe now that the first term vanishes away from the diagonal.
Indeed, if $z\neq \td{z}$ the Gauss Lemma yields $|d_z\r(z,\td{z})|_g=1$, thus the rank of the map $d_z\r(z,\cdot ):\cM\backslash \{z\}\to S^*_z\cM$ is at most $n$. 
Therefore, $\det(\p_{{\td{z}}z}^2\r)=0$ and thus away from the diagonal we have
\begin{equation}\label{eq:simplified_kernel}
	K_{\calN_g}(z,\td{z})=\frac{2
	|\left(\adj(\p_{z{\td{z}}}^2\r)\p_z\r\right)\cdot\p_{{\td{z}}}\r|}{\sqrt{\det g(z)}\sqrt{\det g({\td{z}})}}.
\end{equation}

We first examine $K_{\calN_g}(z,\td{z})$ on $\cM^2$ away from the diagonal when $z\to \p M$ or $\td{z}\to \p M$.
By Proposition \ref{prop:distance_function},  for $z,\td{z}$ away from the diagonal we have
\begin{equation}
	\r(z,\td{z})=\a(x,y,\tx,\ty)-\log(x)-\log(\tx),
\end{equation}
where $\a\in C^\infty \left( M^2\backslash{\Di}\right)$. 
Here without loss of generality we can take $(x,y)$, $(\td{x},\td{y})$ to be global coordinate systems on $\cM$ since $(\cM,g)$ is simple.
Since $\p^2_{z\td{z}}\r=\p^2_{z\td{z}}\,\a$,
  $\adj(\p^2_{z\td{z}}\r)\in C^\infty( M^2\backslash {\Di})$.
% If $\o{g}(z)=x^2g(z)$, ${\o{g}}(\td{z})=\td{x}\,^2g(\td{z})$ then
Moreover, $\sqrt{\det g(z)}={x}^{-n-1}\sqrt{\det \o{g}(z)}$ and $\sqrt{\det g(\td{z})}=\tx\,^{-n-1}\sqrt{\det \o{g}(\td{z})}$ with ${\det \o{g}(z)}, {\det \o{g}(\td{z})}\in C^{\infty}( M)$ and non-vanishing.
Finally, $\p_z\r\in x^{-1}C^\infty(M^2\setminus{\Di})$ and similarly for $\p_{\td{z}}\r$, thus \begin{align}
K_{\calN_g}(z,\td{z})\in x^n\tx\,^n C^\infty\left( M^2\backslash {\Di}\right).
\end{align}

\smallskip

Now we  examine the pullback $\b_0^* K_{\calN_g}$ of $K_{\calN_g}$ to $ M_0^2$.
First let $\calU$ be a neighborhood of $\ff\setminus \p \Di_0$, disjoint from the diagonal and $\rf$.
On $\calU$ we use the projective coordinates \eqref{eq:coord_lf}. 
By Proposition \ref{prop:distance_function}, in $\calU$ we have 	$\b_0^*{\r}=\ta-\log(s),$ $\td{\a}\in C^\infty(\calU).$
Thus the chain rule yields
\begin{equation}
\begin{aligned}
\b_0^*(\p_{x}\r,\p_{y}\r)=&\left(\td{x}^{-1}\left(\p_s\ta-s^{-1}\right),\;\td{x}^{-1}\p_W\ta\right)
=s^{-1}\tx^{-1}{\vphi},
{}\\ 
\b_0^*(\p_{\,\td{x}}\r,\p_{\,\td{y}}\r)=&\left(\p_{\td{x}}\ta-{s}{\td{x}^{-1}}\left(\p_s\ta-s^{-1}\right)-{\td{x}}^{-1}{W^\s}\p_{W^\s}\ta,\;\p_{\td{y}}\ta-{\td{x}}^{-1}\p_W\ta\right)
    =\tx^{-1}\vphi',\label{eq:chain} 
\end{aligned}
\end{equation}
where $\vphi,\, {\vphi'}$ have components in $C^\infty(\calU)$, and further
\begin{equation}\label{eq:mixed_hessian}
\begin{aligned}
	\b_0^*\p^2_{ x\tx}\r
=&{\tx^{-2}}\left(-\p_s\ta+\tx\p^2_{s\tx }\ta-s\p^2_{s}\ta-W^\l\p^2_{sW^\l }\ta\right)
=\tx^{-2}\psi_{00},\\ 
\b_0^*\p^2_{ x \td{y}^\t}{\r}=&{\tx^{-2}}\left(\tx\p^2_{s\td{y}^\t }\ta-\p^2_{sW^\t }\ta\right)
=\tx^{-2}\psi_{0\t},\\
\b_0^*\p^2_{y^\s \ty^\t}{\r}=&{\tx^{-2}}\left(\tx\p^2_{W^\s \ty^\t}\ta-\p^2_{W^\s W^\t}\ta\right)
=\tx^{-2}\psi_{\s\t},\\% 
\b_0^*\p^2_{y^\s\tx}{\r}=
&{\tx^{-2}}\left(-\p_{W^\s}\ta+\tx\p^2_{ W^\s\tx}\ta-s\p^2_{W^\s s}\ta-W^\l\p^2_{W^\s W^\l}\ta\right)% 
=\tx^{-2}\psi_{\s 0},% 
\end{aligned}
\end{equation}
where $\psi_{ij}\in C^\infty(\calU)$. 
Note that $\p_\tx$, $\p_{\ty}$ have different meanings in the left and right hand sides of the above equations.
Since for $H\in \R^{d\times d}$ and $\l\in \R$ we have $\adj(\l H)=\l^{d-1}\adj(H)$, $\b_0^*(\adj(\p_{z\td{z}}^2\r))\in \tx^{-2n}C^\infty(\calU;\R^{(n+1)\times (n+1)})$.
On the other hand, $\b_0^*\sqrt{\det g(\td{z}\,)}=\tx^{-n-1}\td{g}_1$ and $\b_0^*\sqrt{\det g(z)}=s^{-n-1}\tx^{-n-1}\td{g}_2$ with $\td{g}_j\in C^\infty(\calU)$ and non-vanishing for $j=1,2$.
By \eqref{eq:simplified_kernel} we conclude that $\b_0^*K_{\calN_g}\in s^nC^\infty (\calU)$.
This shows that $\b_0^*K_{\calN_g}$ has the claimed behavior 
% near $\rf\cap \ff$ and 
away from $\rf$ and $\Di_0$; moreover, the fact that \eqref{eq:kernel} is symmetric implies that this is also true 
% near $\lf\cap \ff$  and 
away from $\lf$ and $\Di_0$.

\smallskip

We now examine $\b_0^*K_{\calN_g}$ in a neighborhood $\calW$ of a point in $\lf\cap \ff\cap \rf$ away from $\Di_0$.
Near such a point we have $|y-\ty|\neq 0$, hence at least one of the functions $y^\s-\ty\,^\s$ does not vanish.
We may assume without loss of generality that  $y^n-\ty\,^n>0$ and use coordinates
\begin{equation}\label{eq:coordinates_near_all_three}
 r=y^n-\ty\,^n, \: \displaystyle {\th}=\frac{x}{r},\:\displaystyle \td{\th}=\frac{\tx}{r},\:\displaystyle {\widehat{Y}}^{\hat{\l}}=\frac{y^{{{\hat{\l}}}}-\ty\,^{{\hat{\l}}}}{r},\: y,\quad {\hat{\l}}=1, \dots, n-1,	
\end{equation}
which are valid in $\calW$. In terms of these, $r$ is a defnining function for $\ff$ and $\th$, $\td{\th}$ are defining functions for $\lf$, $\rf$ respectively.
A computation using the chain rule yields
\begin{equation}
\begin{aligned}
    \b_0^*\p_x=r^{-1}\p_\th,&\quad \b_0^*\p_{y^\t}=r^{-1}\big((r\p_r-\td{\th}\p_{\,\td{\th}})\d_\t^n+V_\t\big), \quad V_\t\in \calV_b(M_0^2),\; d\td{\th}(V_\t)=d{r}({V}_\t)=0\\*
    \b_0^*\p_{\,\td{x}}=r^{-1}\p_{\,\td{\th}},&\quad 	\b_0^*\p_{\td{y}^\t}=r^{-1}\big(-(r\p_r-{\th}\p_{{\th}})\d_\t^n+\td{V}_\t\big), \quad \td{V}_\t\in \calV_b(M_0^2),\; d{\th}(\td{V}_\t)=d{r}(\td{V}_\t)=0.
\end{aligned}
\end{equation}
Here $ \calV_b(M_0^2)$ denotes smooth vector fields tangent to all faces of $M_0^2$.
In terms of \eqref{eq:coordinates_near_all_three}
$\b_0^*\r=\td{a}-\log(\th)-\log(\td{\th}),\; \td{a}\in C^\infty(\calW),$
so 
\begin{align}
    \b_0^*(\p_x\r,\p_y\r)=(-(r\th)^{-1},0)+r^{-1}\breve{\vphi}, 	\quad 		\b_0^*(\p_{\,\tx}\r,\p_{\,\ty}\r)=(-(r\td{\th}\,)^{-1},0)+r^{-1}\breve{\vphi}\,',\quad\label{eq:first_derivatives}
\end{align}
where $\breve{\vphi}, \breve{\vphi}\,'\in C^\infty(\calW;\R^{n+1})$.
Further,  $(r\p_r-{\th}\p_{{\th}})((r\th)^{-1})=0$ and $(r\p_r-\td{\th}\p_{\td{\th}})((r\td{\th})^{-1})=0$
imply $\b_0^* \p_{z\td{z}\,}^2\r\in r^{-2}C^{\infty}(\calW;\R^{(n+1)\times(n+1)})$, hence $\b_0^*\adj(\p^2_{z\td{z}}{\r})\in r^{-2n}C^\infty(
\calW;\R^{(n+1)\times(n+1)})$.
Noting that $\b_0^*\sqrt{\det g(z)}=r^{-n-1}\th^{-n-1}\breve{g}_1$ and  $\b_0^*\sqrt{\det g(\td{z})}=r^{-n-1}\td{\th}\,^{-n-1}\breve{g}_2
$
with $\breve{g}_j\in C^\infty(\calW)$ and non-vanishing, we use \eqref{eq:simplified_kernel} again and \eqref{eq:first_derivatives} to find that $\b_0^* K_{\calN_g}\in \th^n\td{\th}\,^nC^\infty(\calW)$.
We conclude that $\b_0^*K_{\calN_g}\in C^\infty({M_0^2\setminus \Di_0})$ and vanishes to order $n$ on $\lf$ and $\rf$.

\smallskip

We now examine the behavior of the pullback of \eqref{eq:kernel} by $\b_0 $ near $\p\Di_0$. 
First note that 
 $\b_0^*\r^2\big|_{(M_0^2)^\circ}$ vanishes exactly on $\Di_0\cap (M_0^2)^\circ$. 
By Proposition \ref{prop:distance_function}, $\b_0^*\r^2$ is smooth in a neighborhood of $\Di_0$, hence it also vanishes at $\p\Di_0$; moreover, it vanishes nowhere else on $\ff_p^\circ$:
this follows from  in \cite[Proposition 24]{MR3473907}, according to which for every $p\in \p M$, $\b_0^*\r\big|_{\ff_p^\circ}=\r_{h_p^\ell}(\cdot,e_p)$, where $\r_{h_p^\ell}$ is the hyperbolic distance induced by $h_p^\ell$ on $\ff_p^\circ$.
We now use a variant of the coordinates given by \eqref{eq:coord_lf} near $\Di_0$ and away from $\rf$; we use $(\td{z},Z)=(\td{z},(z-\td{z})/\tx)=(\tx,\ty,s-1,W)$.
In terms of those coordinates $\tx$ is a defining function for $\ff$ and $\Di_0$ is expressed as $\{Z=0\}$.
Observe that on $(M_0^2)^\circ$ one has
\begin{equation}
  \begin{aligned}
  \b_0^*\r^2\big|_{Z=0}=\b_0^*\big(\r^2\big|_{\{z=\td{z}\}}\big)=0,
  \quad\p_{Z^j}(\b_0^*\r^2)\big|_{Z=0}=\tx\b_0^*\left(\p_{{z}^j}(\r^2)\big|_{\{z=\td{z}\}}\right)=0,&\\
\p_{Z^iZ^j}(\b_0^*\r^2)\big|_{Z=0}=\tx^2\b_0^*\left(\p_{{z}^i{z}^j}(\r^2)\big|_{\{z=\td{z}\}}\right)=2\tx^2g_{ij}(\td{z})=2\o{g}_{ij}(\td{z}).\label{eq:partials}\qquad&
\end{aligned}
\end{equation}
\noindent By the smoothness of $\b_0^*\r^2$ near $\Di_0$, \eqref{eq:partials} holds all the way to the front face.
Thus by Taylor's Theorem, viewing $\td{z}$ as parameters, we write
\begin{align}\label{eq:distance_near_diag}
\b_0^*\r^2=&\o{g}_{ij}(\td{z})Z^iZ^j+b_{klm}(\td{z},Z)Z^kZ^lZ^m
\end{align}
where $b_{klm}(\td{z},Z)$ is smooth.
We now show that the expression
	$\hat{K}(z,\td{z}):=\frac{2|\det(\p_{z\td{z}}\r^2/2)|}{\sqrt{\det g(z)}\sqrt{\det g(\td{z})}}$
pulls back to a non-vanishing smooth function in a neighborhood of $\Di_0$, all the way to the front face.
Using computations as in \eqref{eq:mixed_hessian} with $\r$  replaced by $\r^2$, one can 
 conclude that $\tx^2\b_0^*\p_{z
 \td{z}}^2({\r}^2/2)$ is a smooth matrix valued function in a neighborhood of $\p\Di_0$ (note here that its behavior near $s=0$ is irrelevant for this computation). 
 Also, for $z,\td{z}\in \cM$ one has $\big|\det(\p_{z\td{z}}\r^2/2)|_{z=\td{z}}\big|=|\det g(\td{z})|$.
 Therefore, since $\tx^{2n+2}\b_0^*\big(\sqrt{\det g(\td{z})}\sqrt{\det g(\td{z}\,)}\,\big)$ is smooth and non-vanishing near $\Di_0$, $\b_0^*\hat{K}$ is smooth in a neighborhood of the lifted diagonal. 
 Moreover, $\hat{K}\big|_{\Di\cap \cM^2}\equiv 2$ implies that $\b_0^*\hat{K}\big|_{\Di_0}\equiv 2$.

The facts in the preceding paragraph together with
  a standard argument involving the Fourier transform imply that $\b_0^*K_{\calN_g}\in I^{-1}(M_0^2,\Di_0)$,  with principal symbol $\s_0^{-1}(\calN_g)=C_n|\td{\xi}|^{-1}_{\o{g}}$ for $\td{\xi}\neq 0$, where $\td{\xi}$ is the fiber variable for $N^*\Di_0$; near the zero section the principal symbol is smooth and modifying it in compact subsets of the fiber does not change the operator modulo $\Psi_0^{-\infty}(M)$.
Using the identification of $N^*\Di_0$ with ${}^0T^*M$ and the fact that the latter is trivialized by $\{dz^j/x\}$ near $\pM$, we can write invariantly $\s_0^{-1}(\calN_g)(z,\xi)=C_n|{\xi}|^{-1}_{{g^{-1}}}$, $(z,\xi)\in {}^0T^*M\setminus 0$; this agrees with the principal symbol computed in \cite{MR2068966}. 
Since $g$ defines a non-degenerate quadratic form in the fibers of ${}^0T^*M$, $\s_0^{-1}(\calN_g)$ is invertible on ${}^0T^*M$.
We have thus shown that $\calN_g\in \Psi_0^{-1,n,n}(M)$ and is elliptic, completing the proof.
\end{proof}

By Propositions \ref{prop:pseudodifferential} and \ref{prop:boundedness} it follows immediately that for $s\geq 0$
$$\calN_g:x^\d H_0^s(M;\O_0^{1/2})\to x^{\d'} H_0^{s+1}(M;\O_0^{1/2})$$ is bounded if $\d>-n/2$, $\d'<n/2$ and $\d'\leq \d$.
We can now prove a boundedness property for the X-ray transform showing that one can extend it to larger weighted $L^2$ spaces than the ones that appeared in Section \ref{sec:geodesic_flow}. 
We use notations as in Lemma \ref{lm:sobolev_restriction}.

\begin{corollary}\label{cor:boundedness_I}
Let $(\cM^{n+1},g)$ be a simple AH manifold.
If $\d'< \d,\;$ $\d'<0$ and $\d>-n/2$ the X-ray transform is bounded:
\begin{equation}
	I:x^\d L^2(M,dV_g)\to \lg \eta\rg^{-\d'}_hL^2(\p_-S^*M,d\l_\p).
\end{equation}
\end{corollary}
\begin{proof}
We will show that for $\d$, $\d'$ as in the statement  there exists a constant $C$ such that for any $f\in C_c^\infty(\cM)$ one has
\begin{equation}\label{eq:mapping_of_I}
	\|If\|_{x^{\d'} L^2(S\cM;d\l)}\leq C\|f\|_{x^\d L^2(M,dV_g)}.
\end{equation}
Since for $f\in C_c^\infty(\cM)$ one has $If\in  C_{\oX}^\infty(S^*M)\cap {x^{\d'}L^2(S^*\cM;d\l)}$ as explained in the proof of Lemma \ref{lm:sobolev_restriction}, the latter applies for $If$ to show that if  \eqref{eq:mapping_of_I} is known then one has $\|If\|_{\lg \eta\rg^{-\d'}_hL^2(\p_-S^*M,d\l_\p)}\leq C\|f\|_{x^{\d} L^2(M,dV_g)}$, yielding the result by density.

First let $0<\e<\min\{\d+n/2,\d-\d'\}$ and note that for each fixed $\e$ the expression $I(x^{2\e})(z,\xi)=\int_\R x^{2\e}\circ\phi_t(z,\xi)dt$ is uniformly bounded on $S^*\cM$, by the proof of Lemma \ref{lm:sobolev_restriction}.
 Now for $f\in C_c^\infty(\cM)$ apply Cauchy-Schwarz to find
 \begin{equation}
\begin{aligned}
  \|If\|&_{x^{\d'} L^2(S\cM;d\l)}^2=\int_{S^*\cM}x^{-2\d'}|I f(z,\xi)|^2d\l=\int_{S^*\cM}x^{-2\d'}\Big|\int_\R  f(\phi_t(z,\xi))dt\Big|^2d\l\\
\leq &\int_{S^*\cM}x^{-2\d'}\int_\R x^{2\e}\circ\phi_t(z,\xi)dt\int_\R  |(x^{-\e}f)(\phi_t(z,\xi))|^2dt\;d\l\\
&\leq C\int_{S^*\cM}x^{-2\d'}\int_\R |(x^{-\e}f)(\phi_t(z,\xi))|^2dt\;d\l
= C\int_{\cM}x^{-2\d'}\int_{S_{z}^*\cM}I\big(x^{-2\e}|f|^2\big)d\mu_g\;dV_g(z)\\
&\quad=C \|\calN_g(x^{-2\e}| f|^2)\|_{x^{2\d'}L^1(M,dV_g)}.
\end{aligned}   
 \end{equation}
Now if $\d'':=\d-\e$ the choice of $\e$, $\d$ and $\d'$ imply that $2\d'<0$, $2\d''>-n$ and $2\d'\leq 2\d''$. 
On the other hand, an argument similar (but simpler) to the one of 
used in \cite{MR1133743} to show
Proposition \ref{prop:boundedness}, (also see   \cite[Appendix 1]{He_2019}), shows that if $P\in \Psi_{0}^{-1,\calE}(M)$ with $\Re(E_\ell)>n+\s'$, $\Re(E_r)>-\s$ and $\s'-\s\leq \Re(E_f)$ then $P:x^\s L^1(M,dV_g)\to x^{\s'} L^1(M,dV_g)$ is bounded.
Hence the fact that $\calN_g\in \Psi_0^{-1,n,n}(M)$ implies that
\begin{equation}
	\|\calN_g(x^{-2\e}| f|^2)\|_{x^{2\d'}L^1(M,dV_g)}\leq C\|x^{-2\e}| f|^2\|_{x^{2\d''}L^1(M,dV_g)}
=C\| f\|^2_{x^{\d}L^2(M,dV_g)}
\end{equation}
and this finishes the proof.
\end{proof}

\begin{remark}
	By Corollary \ref{cor:boundedness_I} and \eqref{eq:adjoint} one also has that $I^*:\lg \eta\rg^{\d'}_hL^2(\p_-S^*M,d\l_\p)\to x^{-\d} L^2(M,dV_g)$ is bounded for  $\d'< \d,\;$ $\d'<0$ and $\d>-n/2$.

\end{remark}

\section{The Model  Operator}
\label{sec:model_operator}

In this section we show that the model operator of $\calN_g$ at a point $p\in \p M$ can be identified with the normal operator $\calN_{h}$ on the Poincar\'e hyperbolic ball $(\B^{n+1},h)$.
This operator was studied in \cite{MR1104811} and an explicit inversion formula was computed for it using the spherical Fourier transform; using this formula we will show that $\calN_h^{-1}\in\Psi_0^{1,n+1,n+1}(\o{\B^{n+1}})$.
In what follows we always assume that  a choice of coordinates has been made with respect to a point of interest $p\in \pM$, such that the hyperbolic metric $h_p$ in \eqref{eq:hyperbolic_metric} takes the form $h_p=u^{-2}({du^2+|dw|^2})$ with respect to induced linear coordinates $(u,w)$ on ${T_p^+M} $.

The following is an analog of Proposition 2.17 in \cite{MR916753}, which shows that for each $p\in \pM$ the model operator of the Laplacian corresponding to an AH metric $g$ on $\cM$ is the hyperbolic Laplacian on $(T_p^+M,h_p)$:

\begin{proposition}\label{prop:normal_operator}
	For any $p\in \p M$ the model operator $N_p(\calN_g)$ on ${T_p^+M} $ is given by $\calN_{h_p}$, the normal operator corresponding to the X-ray transform on $({T_p^+M},h_p) $.
\end{proposition}
\begin{proof}
We will show that in coordinates $(\td{x},\td{y},s,W)$ (see \eqref{eq:coord_lf}) we have 
\begin{equation}\label{eq:restriction_model_op}
\begin{aligned}
  \b_0^* \big(K_{\calN_g}(z,\td{z})\cdot &\g_0(z)\otimes \g_0(\td{z})\big)\big|_{\ff_p^\circ}\\* 
  =&
  \frac{2|\det(\p^2_{q\td{q}}\,\r_{h_p^\ell}^2/2)|}{\r_{h_p^\ell}^{n}\;\pi_\ell^*\sqrt{\det h_p^\ell}\;\pi_r^*\sqrt{\det h_p^\ell}}\Bigg|_{(q,\td{q})=((s,W),(1,0))}\cdot\left|\frac{dsdW}{s}\right|^{1/2}\left|\frac{dudw}{u^{n+1}}\right|^{1/2}
  \end{aligned}
\end{equation}
with the interpretation of \eqref{eq:restriction2}, where $\r_{h_p^\ell}(q,\td{q})$ denotes the $h_p^\ell$-distance function on $\ff_p^\circ$.
This will imply that in linear coordinates $v=(u,w)$, $\td{v}=(\td{u},\td{w})$ on $T_p^+M$ we have as in \eqref{eq:right_action}
\begin{align}
  N_p(\calN_g)(f\cdot \g_p)(v)=&\int_{T_p^+M}\frac{2|\det(\p^2_{v\td{v}}\r_{h_p}^2/2)|}{\r_{h_p}^{n}\;\pi_\ell^*\sqrt{\det h_p}\;\pi_r^*\sqrt{\det h_p}}
  \Bigg|_{\big((\frac{u}{\td{u}},\frac{w-\td{w}}{\td{u}}),(1,0)\big)}
  f(\td{u},\td{w})\frac{|d\td{u}d\td{w}|}{\td{u}^{n+1}}\cdot \g_p\\
  =&\int_{T_p^+M}\frac{2|\det(\p^2_{v\td{v}}\r_{h_p}^2/2)|}{\r_{h_p}^{n}\;\pi_\ell^*\sqrt{\det h_p}\;\pi_r^*\sqrt{\det h_p}}\Bigg|_{(v,\td{v})}f(\td{v})dV_{h_p}(\td{v})\cdot \g_p=\calN_{h_p}(f\cdot \g_p).
\end{align}
To see the second equality, it suffices to show that $\frac{2|\det(\p_{v\td{v}}\r_{h_p}^2/2)|}{\r_{h_p}^{n}\;\pi_\ell^*\sqrt{\det h_p}\;\pi_r^*\sqrt{\det h_p}}\Big|_{(u,w,\td{u},\td{w})}$ is invariant under the transformations $(u,w,\td{u},\td{w})\mapsto (u,w+a,\td{u},\td{w}+a)$, $a\in \R^n$, and $(u,w,\td{u},\td{w})\mapsto (\l u,\l w,\l\td{u},\l\td{w})$, $\l\in \R$. 
A function $F(v,\td{v})$ on $(\R^+\times \R^n)^2$ has this property exactly when it is annihilated by $\p_{w^\t}+\p_{\td{w}^\t}$ and $u\p_u+w^\s\p_{w^\s}+\td{u}\p_{\td{u}}+\td{w}^\t\p_{\td{w}^\t}$; the fact that $[u\p_u,u\p_{w^\s}]=u\p_{w^\s}=-[w^\t\p_{w^\t},u\p_{w^\s}]$ %
 can be used to show that if $F$ has this property, the same is true of $u\p_v F $, $\td{u}\p_{\td{v}}F$.
By the explicit formula $\cosh \r_{h_p}((u,w),(\td{u},\td{w}))=1+\frac{|w-\td{w}|^2+|u-\td{u}|^2}{2u\td{u}}$ for the $h_p$-distance on $T_p^+M$ when $h_p=u^{-2}({du^2+|dw|^2})$, $\r_{h_p}$ has the required invariance, thus $u\td{u}\p^2_{v\td{v}}\r_{h_p}^2/2$ also does, and the same is true of $(u^{n+1}\td{u}^{n+1}\r_{h_p}^{n}\;\pi_\ell^*\sqrt{\det h_p}\;\pi_r^*\sqrt{\det h_p})$.

We now show \eqref{eq:restriction_model_op}. By  \cite[Proposition 24]{MR3473907}, $\b_0^*\r\big|_{\ff_p^\circ}(q)=\r_{h_p^\ell}(q,e_p)$.
As mentioned earlier, it can be arranged that 
in terms of coordinates $(s,W)$ on $\ff_p^\circ$ as in \eqref{eq:coord_lf} we have  $h_p^\ell=s^{-2}(ds^2+|dW|^2)$, hence again
 $\rho_{h_p^\ell}\big((s,W),(\td{s},\td{W})\big)=\rho_{h_p^\ell}\big((s/\td{s},(W-\td{W})/\td{s}),(1,0)\big)$. Using this, one checks that at $(q,e_p)=((s,Y),(1,0))$ we have
  $\p_{\td{s}}\,\rho_{h_p^\ell}=-s\p_s\rho_{h_p^\ell}-W^\l \p_{W^\l}\rho_{h_p^\ell}
    ,$ $\p_{\td{W}^\l}\rho_{h_p^\ell}=-\p_{W^\l}\rho_{h_p^\ell}.$
Those facts together with \eqref{eq:chain}, \eqref{eq:mixed_hessian} yield that at $q=(s,W)$ 
\begin{align}
\b_0^*\big(\tx(\p_\tx\r,\p_{\ty}\r)\big)\big|_{\ff_p^\circ}=(\p_{\td{s}}\rho_{h_p^\ell},\p_{\td{W}}\rho_{h_p^\ell})\Big|_{(\,\cdot\,,e_p)},
&\quad
\b_0^*\big(\tx(\p_{x}\r,\p_{y}\r)\big)\big|_{\ff_p^\circ}=(\p_s \rho_{h_p^\ell},\p_W\rho_{h_p^\ell})\Big|_{(\,\cdot\,,e_p)},\\
  \b_0^*(\tx^2\p^2_{\tx x}\r)\big|_{\ff_p^\circ}=\p_{s\td{s}}\rho_{h_p^\ell}\Big|_{(\,\cdot\,,e_p)},& \quad 
  \b_0^*(\tx^2\p^2_{\ty^\s x}{\r})\big|_{\ff_p^\circ}=\p^2_{s\td{W}^\s}\rho_{h_p^\ell}\Big|_{(\,\cdot\,,e_p)},\\
\b_0^*(\tx^2\p^2_{\ty^\s y^\t}{\r})\big|_{\ff_p^\circ}= \p^2_{W^\s\td{W}^\t} \rho_{h_p^\ell}\Big|_{(\,\cdot\,,e_p)},
&\quad
 \b_0^*(\tx^2\p^2_{\tx y^\t}{\r})\big|_{\ff_p^\circ}=\p^2_{\td{s}\,W^\s} \rho_{h_p^\ell}\Big|_{(\,\cdot\,,e_p)},
\end{align}
hence by \eqref{eq:determinant}, $\b_0^*\big(\tx^{2n+2}|\det(\p^2_{z\td{z}}\r_{h}^2/2)|\big)\big|_{\tx=0}=|\det(\p_{q\td{q}}^2\r_{h_p^\ell}^2/2)|\big|_{(\,\cdot\,,e_p)}$.
On the other hand, we have $\b_0^*\big(\td{x}^{2n+2}\sqrt{\det g(z)\det g(\td{z})}\,\big)\big|_{\tx=0}={s^{-n-1}}=
\pi_\ell^*\sqrt{\det h_p^\ell}\;\pi_r^*\sqrt{\det h_p^\ell}\Big|_{((s,W),(1,0))}$.
Since $\pi_\ell^*\g_0\otimes\pi_r^* \g_0\big|_{(p,p)}$ can be identified with
$\left|\frac{dsdW}{s}\right|^{1/2}\left|\frac{dudw}{u^{n+1}}\right|^{1/2}$ as explained in Section \ref{ssec:model_operator_background}, we obtain \eqref{eq:restriction_model_op}.
\end{proof}

 As mentioned in \sectionsymbol \ref{ssec:model_operator_background}, for each $p\in \p M$ the model operator $\calN_{h_p}$ can be equivalently realized as an operator on the Poincar\'e ball $(\B^{n+1},h=\frac{4|dz|^2}{(1-|z|^2)^2})$ . 
The following proposition is essentially an immediate consequence of the results in \cite{MR1104811}.
Henceforth we write $\B$ instead of $\B^{n+1}$ (i.e. without a superscript for the dimension).

\begin{proposition}\label{prop:normal_operator_space}
	For any $p\in \p M$ the model operator $N_p(\calN_g)$ can be identified with the operator 
	$\calN_{h}:C_c^\infty(\B;\O_0^{1/2})\to C^{-\infty}({\B};\O_0^{1/2})$ on $(\B,h)$, which for $\d\in (-n/2,n/2)$
	extends continuously to an operator 
	$	\calN_{h}:x^\d L^2(\o{\B};\O_0^{1/2})\to x^\d H_0^1(\o{\B};\O_0^{1/2})$.
	The operator $\calN_h$ has a two sided inverse 
	$\calN_{h}^{-1}\in \Psi_0^{1,n+1,n+1}\big(\o{\mathbb{B}}\big)$ such that $\calN_{h}^{-1}\calN_{h}=\calN_{h}\calN_{h}^{-1}=Id$ on $x^\d L^2(\o{\B};\O_0^{1/2})$ for $\d\in (-n/2,n/2)$.
\end{proposition}
\begin{proof}
	For each $p\in \p M $, $N_p(\calN_g)=\calN_{h_p}$ on $(T_p^+M,h_p)$ by Proposition \ref{prop:normal_operator}, and $\calN_{h_p}$ can be identified with $\calN_h$ on $(\o{\B},h)$ as explained before.
	Thus by Proposition \ref{prop:pseudodifferential}, $\calN_{h}\in \Psi_0^{-1,n,n}(\o{\B})$ and the extension statement follows from Proposition \ref{prop:boundedness}.
	It was observed in \cite{MR1104811} that $\calN_{h}$ can be expressed as
	\begin{equation}\label{eq:convolution_R}
		\calN_{h}f(z)=\int_{\B}\calR\big(\r_{h}(z,\td{z}\,)\big)f(\td{z}\,)dV_{h}(\td{z}\,),\quad f\in C_c^\infty(\B),
	\end{equation}
	where $\r_{h}$ is the geodesic distance function with respect to the hyperbolic metric and
$   \calR(r)=2\sinh^{-n}(r).$
 \eqref{eq:convolution_R} can be interpreted as convolution by a locally integrable radial function on the homogeneous space 
$\mathbb{B}\cong G/H_o$, where $G=O^+(1,n+1)$ is the isometry group of $\B$ and $H_o\cong  O(n+1)$ is the isotropy group of the origin $o\in \mathbb{B}$.
  For $f\in C_c^\infty(\B)$
	\begin{align}
		\calN_{h}f(gH_o)=&\calR*f(gH_o)
		\label{eq:convolution_homog}=\int_{\B}\calR(\td{g}^{-1}gH_o)f(\td{g}H_o)dV_{h}(\td{g}H_o),\quad z=gH_o,\; \td{z}=\td{g}H_o,
	\end{align}
	where above and in what follows, slightly abusing notation, we identify radial functions and distributions on $\B$ with ones on $[0,\infty)$, writing for instance $\calR(z)=\calR(\r_{h}(z,o))$ for $z\in \mathbb{B}$.

	An exact left inverse for $\calN_{h}$ is computed in \cite[Theorems 4.2, 4.3, 4.4]{MR1104811}: one has
	 \begin{equation}
		C_np(\Delta)S_n \;\calN_{h}=Id\quad\text{ on }C_c^\infty(\B),
	\end{equation}
  where $\Delta$ denotes the hyperbolic Laplacian with principal symbol $-|\xi|_{h}^2$,
 $C_n$ is an explicit constant, $p(t)=-(t+n-1)$ and $S_n$ is given by convolution with the locally integrable radial kernel
	\begin{equation}\label{kernel_cases}
		\calS_n(r)=\begin{cases}\coth(r)-1,& n=1
	\\\sinh^{-n}(r)\cosh(r),& n\geq 2\end{cases};
	\end{equation}
	that is,
	\begin{equation}\label{eq:convolution}
		S_nf(z)=\calS_n *f(z)=\int_{\B}\calS_n\big(\r_{h}(z,\td{z}\,)\big)f(\td{z}\,)dV_{h}(\td{z}\,),\quad f\in C_c^\infty(\B).
	\end{equation}

	The fact that $C_np(\Delta)S_n$ is also a right inverse for $\calN_{h}$ follows by tracing through the proofs of Theorems 4.2-4.4 in \cite{MR1104811}. 
	The authors use the spherical Fourier transform of a radial distribution, given by
	$\widehat{f}(\l)=\int_{\B}f(\td{z}\,){\pointyphi}_{-\l}(\td{z}\,)dV_{h}(\td{z}\,)$ for $\l\in \R$, where $\pointyphi_\l$ is the radial eigenfunction of $\Delta$ with eigenvalue $-n^2/4-\l^2$ that satisfies $\pointyphi_\l(o)=1$.
	The spherical Fourier transform is well defined pointwise whenever $f(\td{z}){\pointyphi}_{-\l}(\td{z})$ is integrable; the reason why the formula corresponding to $n=1$ in \eqref{kernel_cases} differs from the one corresponding to $n\geq 2$ is exactly to ensure that $\widehat{\calS_1}$ is well defined.
	Their strategy is to show that $({C_n p(\Delta)\calS_n*\calR})\widehat{\phantom{\d}}=\widehat{\d}$, where $\d$ is the delta distribution at the origin.
	Thus the claim reduces to showing that $({\calR*C_n p(\Delta)\calS_n})\widehat{\phantom{\d}}=\widehat{\d}$.
	This in turn follows from the fact that for radial distributions $\calU$, $\calV$ one has $\widehat{\calU*\calV}(\l)=\widehat{\calU}(\l)\widehat{\calV}(\l)$, and also $\widehat{p(\Delta)\calU}(\l)=-p(-n^2/4-\l^2)\widehat{\calU}(\l)$, provided the expressions make sense.

	Now let $\phi(x)\in C_c^\infty([0,\infty))$ be identically 1 on $[0,1]$ and identically 0 on $[0,2]^c$ and let
	\begin{align}
	S_{n;1}f(z)=&\int_{\B}\phi(\r_{h}(z,\td{z}\,))\calS_n\big(\r_{h}(z,\td{z}\,)\big)f(\td{z}\,)dV_{h}(\td{z}\,)\\
	\text{ and }S_{n;2}f(z)=&\int_{\B}\big(1-\phi(\r_{h}(z,\td{z}\,))\big)\calS_n\big(\r_{h}(z,\td{z}\,)\big)f(\td{z}\,)dV_{h}(\td{z}\,)\text{ for } f\in C_c^\infty(\B^n),	
	\end{align}
	so that $S_n=S_{n;1}+S_{n;2}$.
	Using Proposition \ref{prop:distance_function} one sees that the Schwartz kernel of $S_{n;1}$ vanishes identically near the left and right faces of the 0-stretched product $(\o{\mathbb{B}})_0^2$; thus the $O(r^{-n})$ leading order singularity of $S_n(r)$ at $r=0$ implies that $S_{n;1}\in \Psi_0^{-1}(\o{\mathbb{B}})$ and hence $p(\Delta)S_{n;1}\in \Psi_0^{1}(\o{\mathbb{B}})$ since $p(\Delta)\in \Diff_0^2(\o{\mathbb{B}})$.

	The hyperbolic Laplacian acting on radial distributions is given in terms of geodesic polar coordinates by $\Delta=\p_r^2+n\coth(r)\p_r$ and one checks that for $n\geq 1$
	\begin{equation}\label{eq:lowering_order}
		p({\Delta})\calS_n(r)=-(\Delta+n-1)\sinh^{-n}(r)\cosh(r)=-n\sinh^{-n-2}(r)\cosh(r).	
	\end{equation}
		Since for $f\in C_c^\infty(\B)$
	\begin{align}\label{eq:kernel_of_S_2}
		p(\Delta)S_{n;2}f(z)=\int_{\B}p({\Delta})\big((1-\phi(r))\calS_n(r)\big)\big|_{r=\rho_{h}(z,\td{z})}f(\td{z})dV_{h}(\td{z}),
	\end{align}
	\eqref{eq:lowering_order} and Proposition \ref{prop:distance_function} imply that  $p(\Delta)S_{n;2}\in \Psi_0^{-\infty,n+1,n+1}(\o{\mathbb{B}})$.
	We conclude that $p(\Delta)S_n\in \Psi_0^{1,n+1,n+1}(\o{\mathbb{B}})$ for all $n\geq 1$
	and the spaces on which the inversion is valid follow again from Proposition \ref{prop:boundedness} by density.	
\end{proof}

\section{Parametrix construction and Stability Estimates}

\label{sec:parametrix}

We begin this section with the construction of a parametrix for $\calN_g$:

\begin{proposition}\label{prop:parametrix}
	Let $(\cM^{n+1},g)$ be a simple AH manifold.
	There exists an operator $B$ such that for $\d\in (-n/2,n/2)$ and  $s\geq 0$
	\begin{equation}\label{eq:boundedness_operators}
		B:x^{\d} H_0^{s+1}(M;\O_0^{1/2})\to x^\d H_0^{s}(M;\O_0^{1/2})
	\end{equation}
	is bounded and on $x^\d H_0^{s}(M;\O_0^{1/2})$ one has
\begin{equation}\label{eq:error_term}
	B \calN_g=Id-K, \quad K\in \Psi_0^{-\infty,\calF}(M),\quad F_f\geq 1, \quad F_\ell, F_r\geq n.
 \end{equation}
In particular, $ K:x^\d H_0^s(M;\O_0^{1/2})\to x^\d H_0^s(M;\O_0^{1/2})$ is compact for such $\d$ and $s$.
\end{proposition}

\begin{proof}
We write $\calN_g=A_1+A_2$, where $A_1\in \Psi_0^{-1}(M)$, $A_2\in \Psi_0^{-\infty,n,n}(M)$.
By the ellipticity of $\calN_g$ (and hence of $A_1$),  \cite[Theorem 3.8]{MR1133743} implies the existence of $B_1\in \Psi_0^1(M)$ such that \begin{equation}
	B_1 A_1=Id-K_1,\quad K_1\in \Psi_0^{-\infty}(M).
\end{equation}
Note that $K_1$ is not compact on any weighted Sobolev space $x^\s H_0^s(M;\O_0^{1/2})$ since its kernel does not vanish at $\ff$.
Using Proposition \ref{prop:composition} we reach 
\begin{equation}
	B_1 \calN_g=Id-K_2, \quad K_2:=K_1-B_1A_2\in \Psi_0^{-\infty,n,n}(M). 
\end{equation}
We now improve the error term  to ensure that its kernel  vanishes at the front face.
For each $p\in \p M$ we have  $N_p(K_2)\in \Psi_0^{-\infty,n,n}(\o{\B})$, 
  under the identification of $ ({T_p^+M},h_p)$   with $(\B,h)$, according to the remarks following \eqref{eq:right_action};
 this identification depends smoothly on $p$.
Propositions \ref{prop:normal_operator_space} and \ref{prop:composition} imply that $N_p(K_2)\calN_{h}^{-1}=N_p(K_2)N_p(\calN_g)^{-1}\in \Psi_0^{-\infty,\calE}(\o{\B})$, $E_\ell,$ $E_r\geq n$, $\Re(E_f)\geq 0$.
In fact, we can obtain an improvement of the expansions:
to see this, use Propositions \ref{prop:homomorphism} and \ref{prop:normal_operator_space} to write 
$$N_p(K_2)N_p(\calN_g)^{-1}=N_p(Id-B_1\calN_g)N_p(\calN_g)^{-1}=N_p(\calN_g)^{-1}-N_p(B_1)\in \Psi_0^{1,n+1,n+1}(\o{\B}).$$
Thus $N_p(K_2)N_p(\calN_g)^{-1}\in \Psi_0^{-\infty,\calE}(\o{\B})\cap \Psi_0^{1,n+1,n+1}(\o{\B})\subset\Psi_0^{-\infty,n+1,n+1}(\o{\B})$.
Again using the identification $ ({T_p^+M},h_p)\leftrightarrow(\B,h) $, the convolution kernel of $N_p(K_2)N_p(\calN_g)^{-1}$ is a polyhomogeneous (in fact smooth) half density in  $\mathcal{A}_{phg}^{n+1,n+1}(\ff_*;{}^\ell\Omega_H^{1/2})\otimes \text{span}_{\C}\{\g_*\}$.
By \eqref{eq:short_exact}, it can be extended off of the front face smoothly  to produce an operator $B_2\in \Psi_0^{-\infty,n+1,n+1}(M)$ such that at each $p\in \p M$, $F_p(B_2)$ agrees with the convolution kernel of $N_p(K_2)N_p(\calN_g)^{-1}$.
By Lemma \ref{lm:agreement} and Proposition \ref{prop:homomorphism} this implies that $F_p(B_2\calN_g)=F_p(K_2)$.
Setting $B=B_1+B_2\in\Psi_0^{1,n+1,n+1}(M)$ and using Proposition  \ref{prop:composition} we find
\begin{align}
	B \calN_g=Id-K,\quad &K\in \Psi_0^{-\infty,\calF}(M),\\
	  F_f=\o{\{(1,0)\}}\cup \o{\{(2n+1,1)\}}, \quad &F_\ell=F_r=\o{\{(n,0)\}}\cup \o{\{(n+1,1)\}}.
 \end{align}

As already stated earlier, by Proposition \ref{prop:boundedness} one has that  for $s\geq 0$,
$\calN_g:x^\d H_0^s(M;\O_0^{1/2})\to x^{\d'} H_0^{s+1}(M;\O_0^{1/2})$ is bounded provided $\d>-n/2$, $\d'<n/2$ and $\d'\leq \d$.
Moreover,  
$B:x^{\d'} H_0^{s+1}(M;\O_0^{1/2})\to x^{\d''} H_0^{s}(M;\O_0^{1/2})$ is bounded provided $\d'>-n/2-1$, $\d''<n/2+1$ and $\d''\leq \d'$.
Hence choosing $\d=\d'=\d''\in (-n/2,n/2)$ we obtain \eqref{eq:boundedness_operators} and \eqref{eq:error_term}.
Moreover, for such choice of $\d$, choose $\td{\d}$ such that $\d<\td{\d}<\min\{n/2,\d+1\}$, and $\td{s}>s$ to guarantee that 
% \begin{equation}\label{eq:mapping_for_K}
	$K:x^\d H_0^s(M;\O_0^{1/2})\to x^{\td{\d}}H_0^{\td{s}}(M;\O_0^{1/2})$
% \end{equation}
is bounded, implying that $K:x^\d H_0^s(M;\O_0^{1/2})\to x^{\d}H_0^s(M;\O_0^{1/2})$ is compact, as claimed.
\end{proof}

Proposition \ref{prop:parametrix} together with Proposition \ref{prop:boundedness} imply that for $\d\in (-n/2,n/2)$ 
\begin{equation}\label{eq:functions_in_nullspace}
	x^\d L^2(M;\O_0^{1/2})\cap \ker \calN_g\subset \bigcap_{m\in \R} x^\d H^m_0(M;\smsec)=:x^\d H_0^\infty(M;\smsec)\subset C^\infty(\cM;\smsec).
\end{equation}
We will now show that elements of $x^\d L^2(M;\O_0^{1/2})\cap \ker \calN_g$ also have polyhomogeneous expansions at the boundary. Henceforth we work with functions as opposed to half densities for convenience.
We start by showing tangential regularity.

\begin{lemma}\label{lm:conormal}
Let $u\in x^\d L^2(M,dV_g)\cap \ker \calN_g$, with $\d\in(-n/2,n/2)$. 
Then \begin{align}
  u\in x^{\d} H^\infty_b(M,dV_g):=\{&u\in x^\d L^2(M,dV_g):Pu\in x^\d L^2(M,dV_g), P\in \Diff_b^m(M), m \geq 0\};
\end{align}
here $\Diff_b^m(M)$ (by analogy with $\Diff_0^m(M)$) stands for differential operators consisting of finite sums of at most $m$-fold products of vector fields in $\calV_b(M)$.
\end{lemma}

\begin{proof}
Smoothness of $u$ in $\cM$ was already remarked in \eqref{eq:functions_in_nullspace}.
We will show that for any $m\geq 0$, if it is the case that
$Pu\in x^\d L^2(M,dV_g)$ for every $P\in \Diff_b^m(M)$,  then $P'u\in x^\d L^2(M,dV_g)$ for every $P'\in \Diff^{m+1}_b(M)$.
Since $u\in x^\d L^2(M,dV_g)$ by assumption, this suffices to prove the lemma.
We will first prove an auxiliary fact, namely that
if  $P\in \Diff_b^{m+1}(M)$, $m\geq 0$, then $Pu$ can be written as a finite sum
\begin{equation}\label{eq:conormal_induction}
  Pu=\sum_j Q_j^{(m)}P_j^{(m)}u,\quad Q_j^{(m)}\in \Psi_0^{-\infty,F_\ell,F_r,F_f'}(M)\text{ and }P_j^{(m)}\in \Diff_b^{m}(M),
\end{equation}
where $F_f'=F_f-1\geq 0$ and $F_f,$ $F_\ell$, $F_r$ are as in Proposition \ref{prop:parametrix}.
Once it has been established, \eqref{eq:conormal_induction} will immediately imply the claim
using Proposition \ref{prop:boundedness}.

We will show \eqref{eq:conormal_induction} via an inductive argument.
For the case $m=0$, first observe that a vector field $V\in \calV_b(M)$ lifts from the left factor of $M^2$ to a vector field on $M_0^2$ of the form $x_f^{-1}\td{V}$, where $\td{V}$ is tangent to all boundary faces of $M_0^2$. 
Thus by Proposition \ref{prop:parametrix} we have $Vu=VKu$ with $V  K\in \Psi_0^{-\infty,F_\ell,F_r,F_f'}(M)$.
Hence $V u\in x^{\d}L^2(M,dV_g)$ by Proposition \ref{prop:boundedness}, which shows \eqref{eq:conormal_induction} for $m=0$.

Now suppose that \eqref{eq:conormal_induction} holds for some fixed $m\geq 0$; we will show that it also holds for $m+1$.
Any operator in $\Diff_b^{m+2}(M)$ can be written as a finite sum of the form $\sum _jV_jP_j$, where $V_j\in \calV_b(M)$ and $P_j\in \Diff_b^{m+1}(M)$.
Thus it suffices to differentiate \eqref{eq:conormal_induction} by $V\in \calV_b(M)$ and show that it has the required form.
We find 
\begin{equation}
	VPu=\sum_j VQ_j^{(m)}P_j^{(m)}u=\sum_j\big(Q_j^{(m)}VP_j^{(m)}u-[Q_j^{(m)},V]P_j^{(m)}u\big).
\end{equation}
By Proposition 3.30 in \cite{MR1133743}, $[Q,V]\in \Psi_0^{-\infty,\calE}(M)$ for $Q\in\Psi_0^{-\infty,\calE}(M)$ and $V\in \calV_b(M).$
Thus  $[Q_j^{(m)},V]\in \Psi_0^{-\infty,F_\ell,F_r,F_f'}(M)$ for all $j$.
Since $VP_j^{(m)},$ $P_j^{(m)}\in \Diff_b^{m+1}(M)$ we obtain \eqref{eq:conormal_induction} for $m+1$, finishing the proof.
\end{proof}

In Proposition \ref{prop:polyhomogeneous} below we will use the Mellin transform to show the existence of polyhomogeneous expansion at the boundary for elements in the nullspace of $\calN_g.$
We briefly recall its definition and main properties.
Below we write $\R^+=(0,\infty)$.
\begin{definition}
	If $f\in C^\infty_c(\R^+)$ and $\z\in \C$ we define the Mellin Transform of $f$ by
	\begin{equation}
		 f _{\calM}(\z)=\int_0^\infty x^\z f(x) \frac{dx}{x}.
	\end{equation}
\end{definition}
By the fact that $ f_\calM(\z)=\calF(f(\exp({\,\cdot\,})))(i\z)$ for $\z$ imaginary, we see that for $f\in C^\infty_c(\R^+)$, $f_{\calM}$ is rapidly decaying along each line $\zeta=\a+i\eta$, as $\R\ni \eta\to\pm\infty$ where $\a\in \R$ is constant.
Moreover, 
the Mellin transform induces an isomorphism
\begin{equation}\label{eq:mellin_isomorphism}
	\calM : x^\d L^2\big(\R^+, \frac{dx}{x}\big)\to L^2(\{\Re(\z)=-\Re(\d)\},|d\zeta|)
\end{equation}
with inverse given by 
\begin{equation}\label{eq:Mellin_inverse}
	u(x)=\frac{1}{2\pi }\int_{\Re(\z)=-\Re(\d) }x^{-\z}u_\calM(\z)|d\zeta|.
\end{equation}
By the Paley-Wiener Theorem, if $u\in x^\d L^2\big(\R^+, \frac{dx}{x}\big)$ and $\supp u\subset [0,1)$ then $u_\calM$ extends to a holomorphic function on the half plane $\{\Re(\z)>-\Re(\d)\}$, uniformly in $L^2(\{\Re(\z)=\a\},|d\zeta|)$ for $\a\geq-\Re(\d)$. 
By analogy with the Fourier transform, we also have $(x\p_x u)_{\calM}(\z)=-\z u_{\calM}(\z)$ on the half plane $\{\Re(\z)\geq-\Re(\d)\}$ provided $\supp u\subset [0,1)$ and $u,x\p_xu \in x^\d L^2(\R^+,\frac{dx}{x})$.
Moreover, if $\phi\in C_c^\infty\big([0,\infty)\big)$ is identically 1 near $0$ then $(x^\d|\log(x)|^k\phi)_{\calM}(\z)$ is holomorphic on the half plane $\{\Re(\z)>-\Re(\d)\}$ for $k$ non-negative integer, and using an integration by parts one sees that it extends meromorphically on $\C$, with a pole of order $k+1$ at $\z=-\d$.
If $M$ is a compact manifold with boundary one can use a product decomposition $[0,\e)_x\times \p M$ of a collar neighborhood of $\p M$ and compute the Mellin transform in the $x$ variable for polyhomogeneous conormal functions supported near $\p M$.
If $\phi\in C^\infty(M)$ is supported near $\p M$ and $u\in \calA_{phg}^E(M)$,
 then $(u\phi(x))_\calM$ is meromorphic on $\C$ with poles of order $p+1$ at $\z=-s-\ell$ and values in $C^\infty(\p M)$ for each $(s,p)\in E$ and for $ \ell\in \mathbb{N}_{ 0}=\{0,1,\dots\}$. 
The fact that the space $\calA_{phg}^E(M)$ is invariantly defined, as already remarked earlier, implies that the analyticity properties of $(\phi u)_\calM$ are invariantly defined.

Before we show the existence of an asymptotic expansion for elements in the nullspace of $\calN_g$ we show a lemma about index sets.
\begin{lemma}\label{lm:index_sets}
Let $E_1$, $E_2$, $F\subset \C\times \mathbb{N}_0$ be index sets. 
Then $(E_1\o{\cup}E_2)+F\subset (E_1+F)\o{\cup}(E_2+F)$.
 
\end{lemma}
\begin{proof}
First note that $(E_1\cup E_2)+F=(E_1+F){\cup}(E_2+F)\subset (E_1+F)\o{\cup}(E_2+F)$. 
Now suppose that $(s,p_1+p_2+1)\in E_1\o{\cup}E_2$, where $(s,p_1)\in E_1$ and $(s,p_2)\in E_2$ and let $(\td{s},\td{p})\in F$. 
Then $(s+\td{s},(p_1+\td{p})+(p_2+\td{p})+1)\in (E_1+F)\o{\cup}(E_2+F)$, so it is also the case that $(s,p_1+p_2+1)+(\td{s},\td{p})=(s+\td{s},p_1+p_2+\td{p}+1)\in (E_1+F)\o{\cup}(E_2+F)$ by \eqref{eq:index_sets_assumptions} and we have shown the claim.
\end{proof}

\begin{remark}
	In general one does not have $(E_1\o{\cup}E_2)+F=(E_1+F)\o{\cup}(E_2+F)$. For instance consider the index sets $E_1=\o{\{(1,10)\}}$, $E_2=\o{\{(1/2,0)\}}$ and $F=\o{\{(1/2,5),(0,0)\}}$. 
	Then $(1,16)\in (E_1+F)\o{\cup}(E_2+F)\setminus \big((E_1\o{\cup}E_2)+F\big)$.
\end{remark}

\begin{proposition}\label{prop:polyhomogeneous}
	Let $u\in x^\d L^2(M,dV_g)\cap \ker \calN_g$, with $\d\in (-n/2,n/2)$.
	Then $u\in \calA_{phg}^E(M)$ with $E=\o{\bigcup}_{j\geq 0}(F_\ell+jF_f)$, where $F_\ell$, $F_f$ are the index sets in \eqref{eq:error_term} and $jF_f=\sum_{i=1}^jF_f$.
	Note that $F_\ell+jF_f\geq n+j$ and hence $E$ is an index set.
\end{proposition}

\begin{proof}
By \eqref{eq:functions_in_nullspace}, any $u$ as in the statement is smooth in $\cM$,
hence it suffices to show the existence of an asymptotic expansion at the boundary for $u$.
We first show that if $u\in x^\d L^2(M,dV_g)\cap \ker \calN_g$  for {some} $\d\in (-n/2,n/2)$ then $u\in x^{\d'} H_0^\infty(M,dV_g)$ for {all} $\d'<n/2$.
Since $u=Ku$, the mapping properties of $K$ (by \eqref{eq:error_term} and Proposition \ref{prop:boundedness})
 imply that $u\in x^{\d_1}H_0^\infty(M,dV_g)$ provided $\d_1<n/2$, $\d_1\leq\d+1$, that is, the existence of a parametrix allows us to obtain an improvement in the decay of $u$.
 Using $j$ times the improved decay and $u=Ku$, we inductively find that $u\in x^{\d_j}H_0^\infty(M,dV_g)$, provided $\d_j<n/2$ and $\d_j\leq \d+j$, so taking $j$ sufficiently large we conclude that $u\in x^{\d'} H_0^\infty(M,dV_g)$ for $\d'<n/2$.
 By Lemma \ref{lm:conormal}, $u\in x^{\d'}H_b^\infty(M,dV_g)$ for $\d'<n/2$, hence $u\in x^{\t}H_b^\infty(M,d\mu_b)$ for $\t<n$; the latter is the same space as in the statement of Lemma \ref{lm:conormal} with $dV_g$ replaced by $d\mu_b$, a measure on $M$ induced by a section of $\Omega_b(M)$, of the form ${x^{-1}|dxdy|}$ locally near $\p M$.
 (In this proof we use this measure due to its more natural behavior with respect to the Mellin transform.)

Functions in
$x^\t H_b^\infty(M,d\mu_b)$ supported near $\p M$ can be identified with functions which lie in $x^\t H^k_b(dx/x;H^\ell(\pM))$ for all $k,\ell\in \mathbb{N}_0$; 
 the latter is the space of $v:\R^+\to H^\ell(\p M)$ which are $k$ times Fr\'echet differentiable almost everywhere on $\R^+$ (and supported near 0), and
$ 	\|x^{-\t}(x\p_x)^jv\|_{H^\ell(\p M)}\in L^2(dx/x)$ for $j=0,\dots,k$.
Therefore, if $\phi\in C_c^\infty(M)$ is supported in a  small neighborhood of $\p M$ and identically 1 near $\p M$, by taking the Mellin transform in $x$ we find that $(\phi u)_\calM(\z)$ is holomorphic in the half plane $\{\Re(\z)>-\t\}$, with values in functions smooth in $y$  {and with the $L^2(\{\Re(\z)=\a\},|d\zeta|)$ norm of $\|(\phi u)_\calM\|_{H^\ell(\p M)}$ being uniformly bounded  for $\a\geq -\t$} for each $\ell$.

We now recover the leading order term in the expansion of $u$ at $\p M$.
Observe that localizing $K$ near the boundary from the left does not alter its index sets: that is,  if $\phi$ is smooth and supported near $\p M$ with $\phi\equiv 1$ near $\p M$, then $\phi K\in \Psi_0^{-\infty,\calF}(M)$, with $\calF=(F_\ell,F_r,F_f)$ as in \eqref{eq:error_term}.
Recall from the proof of Proposition \ref{prop:parametrix} that $F_\ell=\bigcup_{j\geq 0}\o{\{(n+j,p_j)\}}$ with $p_0=0$; denote  $F_\ell^k=\bigcup_{j\geq k}\o{\{(n+j,p_j)\}}$ for $k\in \mathbb{N}_0$.
Now let $P_0=(x\p_x-n)\in \Diff_0^1(M)$;
 $P_0 $ lifts to $M_0^2$ to a $C^\infty $ vector field that takes the form $(s\p_s-n)$ near $\lf$ in terms of coordinates \eqref{eq:coord_lf}.
Then
$K_0:=P_0(\phi K)\in \Psi_0^{-\infty,F_\ell^1,F_r,F_f}(M)$, that is, the term of order $n$ in the expansion of $\phi K$ at the left face of $M_0^2$ is removed.
We can now see that $K_0u\in x^{\t}H_b^\infty(M;d\mu_b)$ for $\t<n+1$:
indeed, since $u\in x^{\t'}H_b^\infty(M;d\mu_b)$ for $\t'<n$, $Pu\in  x^{\t'}L^2(M;d\mu_b)$ for $P\in \Diff_b^m(M)$, $m\geq 0$.
By an inductive argument using commutators as in Lemma \ref{lm:conormal}, one sees that if $P'\in \Diff_b^{m}(M)$, $m\geq 0$, then $P'K_0u=\sum_{j=1}^{J}Q_jP_ju$, where $P_j\in \Diff_b^m(M)$ and $Q_j\in \Psi_0^{-\infty,F_\ell^1,F_r,F_f}(M)$.
Thus Proposition \ref{prop:boundedness}, $F_\ell^1\geq n+1-\e$ for all $\e>0$, and $F_f\geq 1$  imply that if $\t<n+1$ then for $P'\in \Diff_b^{m+1}(M)$ we have $ P'K_0 u\in x^\t L^2(M;d\mu_b)$   for all $m\geq 0$. In other words, $K_0u\in x^{\t}H_b^\infty(M;d\mu_b) $ for such $\t$, as claimed.
This implies that $(K_0 u)_{\calM}(\z)$ is holomorphic on the half plane $\{\Re(\z)>-n-1\}$ with values in functions smooth in $y$.
Since $(P_0(\phi u))_\calM=(-\z-n)(\phi u)_\calM$, $(\phi u)_\calM=(-\z-n)^{-1}(K_0 u)_\calM$ and   we conclude that $(\phi u)_\calM$ extends  meromorphically on the half plane $\{\Re(\z)>-n-1\}$,  with a pole of order $1$ at $\z=-n$ and values in smooth functions on $\pM$.
Computing the inverse Mellin transform on the line $\{\Re(\z)=-n-1+\e\}$, where $\e>0$ is small (note that on such a line $(-\z-n)^{-1}(K_0 u)_\calM$ depends smoothly on $y$ and is in $L^2(\{\Re(\z)=-n-1+\e\};|d\zeta|)$ for each $y$), we recover the leading term of the expansion of $u$: near $\p M$
	\begin{equation}
  u(x,y)=a_{00}(y)x^{n}+v,\quad a_{00}\in C^\infty(\pM),\: v\in x^{\t} H_b^\infty(M,d\mu_b), \:\t<n+1.  
  \end{equation}

Now suppose that for $m\geq1$ we have recovered the asymptotic expansion of $u$ up to the  $m$-th exponent appearing  in the index set $E=\o{\bigcup}_{j\geq 0}(F_\ell+jF_f)$ and corresponding to powers of $x$. It  will be convenient to write  $E=\bigcup_{j\geq 0}\o{\{(n+j,r_j)\}}$, where $r_j$ is the highest power of a logarithmic factor multiplying $x^{n+j}$ in the expansion induced by $E$. Similarly, write $F_\ell+mF_f=\bigcup_{j\geq m}\o{\{(n+j,r_{j;m})\}}$, so that  $r_j+1=\sum_{m=0}^j(r_{j;m}+1)$.
Note that $r_{j;0}=p_j$ and also $r_0=r_{0;0}=p_0=0$. 
So suppose that we have
\begin{equation}
  \label{eq:expansion}
\begin{aligned}
	u=&u_m+v \text{ where }v\in x^{\t} H_b^\infty(M,d\mu_b),\quad \t<n+m,\\
					&\text{ and  near }\p M \quad u_m(x,y)=\sum_{j=0}^{m-1}\sum_{k=0}^{r_j}a_{jk}(y)x^{n+j}|\log x|^k,\quad a_{jk}\in C^\infty(\pM).
\end{aligned}
\end{equation}
We will show that \eqref{eq:expansion} holds for $m+1$; then by induction we will be done.

By \eqref{eq:expansion}, $(\phi u)_\calM$ is meromorphic on the half plane $\{\Re(\z)>-n-m\}$ with poles of order $r_j+1$ at $\z=-n-j$ for $0\leq j\leq m-1$.
Set $P_j=(x\p_x-n-j)\in \Diff_0^1(M)$ and write $K_m:=\prod_{j=0}^{m}P_j^{p_j+1}(\phi K)$; then $K_m\in \Psi_0^{-\infty, F_\ell^{m+1}\!,F_r,F_f}(M)$ as before.
Now by \eqref{eq:expansion}
\begin{equation}\label{eq:mth_term}% 
	\prod_{j=0}^{m}P_j^{p_j+1}(\phi u) = K_mu_m+K_mv.
\end{equation}
Since $K_m\in \Psi_0^{-\infty, F_\ell^{m+1}\!,F_r,F_f}(M)$ with $F_f\geq 1$ and $F_\ell^{m+1}\geq n+m+1-\e$ for all $\e>0$, the fact that $v\in x^{\t} H_b^\infty(M,d\mu_b),$ for $ \t<n+m $ implies that $K_mv\in x^\t H_b^\infty(M,d\mu_b)$ for $\t<n+m+1$ using the same commutator argument as before and Proposition \ref{prop:boundedness}. % 
Moreover, it follows by Proposition \ref{prop:polyhomogeneous_mapping} that $K_mu_m\in \calA_{phg}^G(M)$, where 
\begin{equation}\label{eq:inclusion}
	G=F_\ell^{m+1}\o{\cup}\Big(\o{\bigcup}_{j=0}^{m-1}(F_\ell+jF_f)+F_f\Big){\subset} F_\ell^{m+1}\o{\cup}\Big(\o{\bigcup}_{k=1}^{m}(F_\ell+kF_f)\Big)=:G',
\end{equation}
where the inclusion follows from Lemma \ref{lm:index_sets}.
Thus $K_mu_m\in \calA_{phg}^{G'}(M)$.
Upon  taking the Mellin transform in \eqref{eq:mth_term},
\begin{equation}\label{eq:induictive_step}
	\prod_{j=0}^{m}(-\z-n-j)^{p_j+1}(\phi u)_\calM(\z)=(K_mu_m)_\calM(\z)+(K_mv)_\calM(\z),
\end{equation}
where $(K_mv)_\calM(\z)$ is holomorphic in $\{\Re(\z)>-n-m-1\}$ (with values in $C^\infty(\p M)$).
On the other hand, for $1\leq j\leq m$, $(K_mu_m)_\calM(\z)$ has a pole of order $\sum_{k=1}^j(r_{j;k}+1)$ at $\z=-n-j$.
Note that the index set $F_\ell^{m+1}$ in \eqref{eq:inclusion} does not contribute any poles in the open half plane $\{\Re(\z)>-n-m-1\}$.
Thus upon dividing we find that $(\phi u)_\calM(\z)$ is meromorphic on the half plane $\{\Re(\z)>-n-m-1\}$ with values in $C^\infty(\p M)$ and poles of order $p_j+1+\sum_{k=1}^j(r_{j;k}+1)=(r_{j;0}+1)+\sum_{k=1}^j(r_{j;k}+1)=r_j+1$ at $\z=-n-j$, $0\leq j\leq m$.
Taking the inverse Mellin transform of \eqref{eq:induictive_step} on a vertical line $\{\Re(\z)=-n-m-1+\e\}$ for small $\e>0$ similarly to the first inductive step  we obtain \eqref{eq:expansion} for $m+1$.
\end{proof}

\begin{remark}\label{rem:sharpness}
It follows from \eqref{eq:inclusion} that the index set $E$ in the statement of Proposition \eqref{prop:polyhomogeneous} includes higher powers of logarithmic factors than it needs to, but its form suffices for our needs.
\end{remark}

We will need the following standard result from functional analysis (see \cite{MR2068966} for a proof):
\begin{lemma}
\label{lm:functional}
Let $X$, $Y$, $Z$ be Banach spaces, and let $A:X\to Y$ be bounded and injective.
If there exists a compact operator $K:X\to Z$  such that
\begin{equation}
	\|u\|_X\leq C\left(\|Au\|_Y+\|Ku\|_Z\right), \quad u\in X
\end{equation}
for some constant $C$, then there exists a constant $C'$ such that
\begin{equation}
	\|u\|_X\leq C'\|Au\|_Y, \quad u\in X.
\end{equation}
\end{lemma}

% }
\medskip
We now prove the main theorem:

\begin{proof}[Proof of Theorem \ref{thm:main}]
Let $u\in x^\d L^2(M,dV_g)\cap \ker(\calN_g)$, $\d\in (-n/2,n/2)$.
We claim that $u=0$.
Note that by Corollary \ref{cor:boundedness_I} the X-ray transform is well defined on such a $u$ in the sense that $Iu\in {\lg\eta\rg_h^{-\d'}}{L^2(\p_-S^*M;d\l_\p)}$, $\d'<\min\{\d,0\}$.
By Proposition \ref{prop:polyhomogeneous}, $u\in \calA_{phg}^E(M)$, $E\geq n$.
In particular, $u\in x^\d L^2(M,dV_g)$ for $\d<n/2$.
Now  by \eqref{eq:adjoint_2} and the discussion immediately after it we find
\begin{align}
	0=(\calN_g u,u)_{L^2(M,dV_g)}=&
	({I}^*{I} u,u)_{L^2(M,dV_g)}=\|{I}u\|_{L^2(\p_-S^*M;d
		\l_\p)}^2.
\end{align}
This implies that $Iu=0$.
Then one checks that the proof of Theorem 1 in \cite{2017arXiv170905053G}, which shows injectivity of $I$ on $xC^\infty (M)$, also applies for polyhomogeneous functions in $\calA_{phg}^E(M)$, $E\geq 1$.
More specifically, by the proof of \cite[Proposition 3.15]{2017arXiv170905053G} it follows that for $u\in \calA_{phg}^E(M)\cap \ker I$
  one has the stronger result $u\in \dot{C}^\infty(M)$ (i.e. $u$ vanishes to infinite order at the boundary).
  Then the injectivity argument using Pestov identities in the proof of Theorem 1 in the same paper yields $u\equiv 0$.
We have shown that $\calN_g$ is injective on $x^\d L^2(M,dV_g)$, $\d>-n/2$.
Now by Proposition \ref{prop:parametrix} we have 
\begin{equation}
	\|u\|_{x^\d H_0^{s}(M,dV_g)}\leq C\left(\|\calN_gu\|_{x^\d H_0^{s+1}(M,dV_g)}+\|Ku\|_{x^\d H_0^{s}(M,dV_g)}
	\right), \quad \d\in (-n/2,n/2),\: s\geq 0,
\end{equation}
where $K:x^\d H_0^{s}(M,dV_g)\to x^\d H_0^{s}(M,dV_g)$ is compact.
Thus Lemma \ref{lm:functional} implies
\begin{equation}
	\|u\|_{x^\d H_0^{s}(M,dV_g)}\leq C'\|\calN_gu\|_{x^\d H_0^{s+1}(M,dV_g)}, \quad \d\in (-n/2,n/2),\: s\geq 0,
\end{equation}
which is the {claimed estimate}.
\end{proof}

\subsubsection*{Acknowledgments}
This research was carried out in part during the author's visit at the Department of Mathematics, Stanford University in February 2019, which was partially supported by NSF Grant No. DMS-1608223 of Rafe Mazzeo and partially by a travel fellowship by the Department of Mathematics, University of Washington. 
The author is indebted to Rafe Mazzeo for his hospitality during this visit, for many helpful discussions, and for suggesting the arguments included in Lemma \ref{lm:conormal} and Proposition \ref{prop:polyhomogeneous}. 
The work contained in this paper has appeared as Chapter 3 of the author's University of Washington PhD thesis (\cite{EptaminitakisNikolaos2020Gxto}).
The author would like to thank Robin Graham for numerous helpful discussions, suggestions, and detailed feedback on the aforementioned thesis, and Gunther Uhlmann for helpful discussions and financial support.

{
\printbibliography}
\end{document}